\documentclass[11pt,letterpaper,reqno]{amsart}

\usepackage[margin=1.08in]{geometry}
\usepackage{amsmath,amssymb,amsthm,mathtools,mathrsfs}
\usepackage{microtype}
\usepackage[T1]{fontenc}
\usepackage[utf8]{inputenc}
\usepackage{lmodern}
\usepackage{enumitem}
\usepackage{xcolor}
\usepackage{hyperref}
\usepackage{aliascnt}

\hypersetup{
  colorlinks=true,
  linkcolor=black,
  citecolor=black,
  urlcolor=black,
  pdftitle={Symmetric Kahane--Salem--Zygmund Inequalities and the Supremum Norm},
  pdfauthor={Daniel M. Pellegrino, Anselmo B. Raposo Jr., Eduardo V. Teixeira}
}

\allowdisplaybreaks
\numberwithin{equation}{section}
\setlist{itemsep=0.25em,topsep=0.45em}
\newtheorem{maintheorem}{Theorem}
\newtheorem{theorem}{Theorem}[section]
\newaliascnt{proposition}{theorem}
\newtheorem{proposition}[proposition]{Proposition}
\aliascntresetthe{proposition}
\newaliascnt{lemma}{theorem}
\newtheorem{lemma}[lemma]{Lemma}
\aliascntresetthe{lemma}
\newaliascnt{corollary}{theorem}
\newtheorem{corollary}[corollary]{Corollary}
\aliascntresetthe{corollary}
\theoremstyle{definition}
\newaliascnt{definition}{theorem}

\aliascntresetthe{definition}
\theoremstyle{remark}
\newaliascnt{remark}{theorem}
\newtheorem{remark}[remark]{Remark}
\aliascntresetthe{remark}
\newaliascnt{question}{theorem}

\aliascntresetthe{question}

\newcommand{\R}{\mathbb R}

\newcommand{\sym}{\mathrm{sym}}
\newcommand{\per}{\operatorname{per}}
\newcommand{\distH}{d_{\mathrm H}}

\title[Symmetric KSZ inequalities]{Symmetric Kahane--Salem--Zygmund Inequalities and the Supremum Norm}

\author[D. M. Pellegrino]{Daniel M. Pellegrino}
\author[A. B. Raposo Jr.]{Anselmo B. Raposo Jr.}
\author[E. V. Teixeira]{Eduardo V. Teixeira}

\subjclass[2020]{Primary 46B28, 60G15; Secondary 15A15, 15A60, 05B20, 42C10}
\keywords{Kahane--Salem--Zygmund inequality, symmetric multilinear form, unimodular coefficients, supremum norm, permanent, Walsh analysis, Hadamard matrix, Gaussian width}

\begin{document}

\begin{abstract}
The Kahane--Salem--Zygmund inequality provides unimodular $m$-linear forms with small supremum norm.  We consider its dimension-free symmetric constant $C_m^{\sym}$ over the real scalar field, defined by the requirement that, for every $n$, some real symmetric unimodular $m$-linear form $A:(\ell_\infty^n)^m\to\mathbb R$ satisfy
\[
 \|A\|\le C_m^{\sym}n^{(m+1)/2}.
\]
For complex scalars, Boas obtained under permutation symmetry an upper bound of order at most $\sqrt{m\log m}\,\sqrt{m!}$; over the real scalar field, an elementary argument gives the sharper order $\sqrt m\,\sqrt{m!}$.  We prove
\[
 C_m^{\sym}\ge
 \left(\sqrt{\frac2e}+o(1)\right)\frac{\sqrt{m!}}m,
\]
using the square-free Walsh spectrum of the diagonal polynomial.  In the opposite direction, we establish
\[
 C_m^{\sym}\le C_0\sqrt{m!},
\]
where $C_0$ is absolute, removing the factor $\sqrt m$ from the real upper bound.  This estimate follows from a geometric argument in which rigidity for Gram permanents reduces the relevant configurations to sets controlled by Gaussian width.  For unrestricted unimodular forms, we also obtain a quantitative rectangular estimate from truncated Hadamard matrices; in equal dimension, the normalized minimum is at most $1+o(1)$ whenever $m=o(n^{5/6})$.
\end{abstract}

\maketitle

\begingroup
\hypersetup{linkcolor=black}
\color{black}
\small
\setcounter{tocdepth}{1}
\setlength{\parskip}{0pt}
\vspace{-1.5em}
\tableofcontents
\endgroup

\section{Introduction and main results}
\label{sec:introduction}

Let $m,n_1,\ldots,n_m$ be positive integers.  The real Kahane--Salem--Zygmund problem asks for signs
$\varepsilon_{i_1,\ldots,i_m}\in\{-1,1\}$ for which
\[
 A(x^{(1)},\ldots,x^{(m)})
 =
 \sum_{i_1=1}^{n_1}\cdots\sum_{i_m=1}^{n_m}
 \varepsilon_{i_1,\ldots,i_m}
 x^{(1)}_{i_1}\cdots x^{(m)}_{i_m}
\]
satisfies
\begin{equation}\label{eq:classical-ksz-intro}
 \|A\|
 \le C_m
 \max_{1\le j\le m}n_j^{1/2}
 \prod_{j=1}^m n_j^{1/2}.
\end{equation}
Here all variables are real, and
\[
 \|A\|:=\sup_{\|x^{(j)}\|_\infty\le1\ (1\le j\le m)}
 |A(x^{(1)},\ldots,x^{(m)})|.
\]
An anisotropic form of this estimate follows from
\cite[Proposition~2.3]{AlbuquerqueRezende}, using
$\sum_{j=1}^m n_j\le m\max_j n_j$.  When $n_1=\cdots=n_m=n$, the right-hand side has order $n^{(m+1)/2}$.  The inequality originates in the work of Salem--Zygmund and Kahane.  Multilinear versions and related forms appear in
\cite{BennettGoodmanNewman,Kahane,ManteroTonge,SalemZygmund,Varopoulos}; see also
\cite{DefantMastylo,MastyloSzwedek} for Banach-space formulations.

Denote by $S_{m,\mathbf n}$ the minimum of $\|A\|$ among all real sign forms in dimensions $\mathbf n=(n_1,\ldots,n_m)$, and write $S_{m,n}$ in the equal-dimensional case.  For fixed $m\ge2$, the normalized minima $S_{m,\mathbf n}/(\max_j n_j^{1/2}\prod_j n_j^{1/2})$ are asymptotically bounded above by $1$ as all dimensions tend to infinity; see \cite[Theorem~1.1]{PellegrinoRaposo}.  In fact, only the second-smallest dimension needs to tend to infinity, by \cite[Theorem~3.9]{PellegrinoRaposo}.  Theorem~\ref{thm:unrestricted-summary} gives an explicit quantitative estimate that also allows $m$ to vary.

The symmetric problem is obtained by requiring permutation-invariant coefficients.  In equal dimensions this means
\[
 \varepsilon_{i_{\sigma(1)},\ldots,i_{\sigma(m)}}
 =\varepsilon_{i_1,\ldots,i_m}
 \qquad(\sigma\in S_m).
\]
Let $S_{m,n}^{\sym}$ be the least norm of a form satisfying this condition,
and define
\[
 C_m^{\sym}:=
 \sup_{n\ge1}\frac{S_{m,n}^{\sym}}{n^{(m+1)/2}}.
\]
Boas proved the estimate
\begin{equation}\label{eq:boas-intro}
 S_{m,n}^{\sym}
 \le
 \sqrt{32m\log(6m)}\,\sqrt{m!}\,n^{(m+1)/2};
\end{equation}
see \cite[Theorem~4]{Boas}.  Although \cite[Theorem~4]{Boas} is stated over the complex scalars, the coefficients used there are real signs; restricting the variables to $\mathbb R^n$ gives \eqref{eq:boas-intro}.  Over the real cube, an elementary argument already gives
\[
 C_m^{\sym}\le\sqrt{2m\log2}\,\sqrt{m!}.
\]
Indeed, assign independent signs to the permutation orbits.  If an orbit
$\mathcal O$ has cardinality $d_{\mathcal O}$, then the variance proxy at
each cube configuration is at most
$\sum_{\mathcal O}d_{\mathcal O}^2\le m!n^m$.
By changing the overall sign of each input vector, we may fix its first
coordinate to be $1$ without changing the absolute value of the form.
There are then $N=2^{m(n-1)}$ configurations to test.  The elementary
subgaussian maximum estimate in Lemma~\ref{lem:subgaussian-max} bounds the
expected norm by $\sqrt{2m!n^m\log(2N)}$, which gives the displayed bound
because $[m(n-1)+1]/n\le m$.
Theorem~A removes the factor $\sqrt m$ from this real-cube bound and the
factor $\sqrt{m\log m}$ from the quoted estimate of Boas.

\begingroup
\renewcommand{\themaintheorem}{A}
\begin{maintheorem}[The symmetric constant]\label{thm:main}
There is an absolute constant $C_0>0$ such that, for every $m\ge3$,
\begin{equation}\label{eq:main-upper-numbered}
 C_m^{\sym}\le C_0\sqrt{m!}.
\end{equation}
Moreover, as $m\to\infty$,
\begin{equation}\label{eq:main-lower-numbered}
 C_m^{\sym}\ge
 \left(\sqrt{\frac2e}+o(1)\right)\frac{\sqrt{m!}}m.
\end{equation}
Consequently,
\begin{equation}\label{eq:main-log-numbered}
 \log C_m^{\sym}=\frac12\log(m!)+O(\log m),
 \qquad
 \bigl(C_m^{\sym}\bigr)^{1/m}\sim\sqrt{\frac me}.
\end{equation}
\end{maintheorem}
\endgroup

The two estimates determine the factorial growth of $C_m^{\sym}$ within a multiplicative factor $O(m)$.  The logarithmic asymptotics in \eqref{eq:main-log-numbered} also follow by combining the lower bound with $C_m^{\sym}\le\sqrt{2m\log2}\,\sqrt{m!}$; the new upper bound has a prefactor independent of the degree.  In particular, it remains open whether
$C_m^{\sym}/\sqrt{m!}$ is bounded below by a positive absolute constant.
\begingroup
\color{black}
Section~\ref{sec:block-restrictions} shows that square-free Walsh restrictions, even after identifying the variables in arbitrary blocks, cannot remove the factor $m^{-1}$ obtained from the full diagonal.
\endgroup

For the unrestricted problem, let $\mathcal H$ denote the set of Hadamard orders, including $1$ and $2$, and define
\[
 h_+(n):=\min\{N\in\mathcal H:N\ge n\}.
\]
\begingroup
\color{black}
There is an absolute constant $K_{\mathrm H}>0$ such that
\endgroup
\begin{equation}\label{eq:hadamard-gap-intro}
 h_+(n)-n\le K_{\mathrm H}n^{1/6}
 \qquad(n\ge1).
\end{equation}

\begingroup
\renewcommand{\themaintheorem}{B}
\begin{maintheorem}[The unrestricted comparison]
\label{thm:unrestricted-summary}
For every $m\ge2$ and $1\le n_1\le\cdots\le n_m$,
\begin{equation}\label{eq:unrestricted-rectangular}
 \frac{S_{m,\mathbf n}}
 {n_m^{1/2}\prod_{j=1}^m n_j^{1/2}}
 \le
 \prod_{j=2}^m
 \left(\frac{h_+(n_j)}{n_j}\right)^{1/2}
 \le
 \exp\left\{\frac{K_{\mathrm H}}2
 \sum_{j=2}^m n_j^{-5/6}\right\}.
\end{equation}
In particular,
\begin{equation}\label{eq:two-parameter-main}
 \frac{S_{m,n}}{n^{(m+1)/2}}
 \le
 \exp\left\{\frac{K_{\mathrm H}}2(m-1)n^{-5/6}\right\}.
\end{equation}
Consequently, the normalized unrestricted quantity is at most $1+o(1)$
whenever $m=o(n^{5/6})$.
\end{maintheorem}
\endgroup

If $n$ is a Hadamard order, the first estimate in \eqref{eq:unrestricted-rectangular} gives
$S_{m,n}\le n^{(m+1)/2}$.  Combining this with the lower estimate for the symmetric problem at $n\asymp m^2$ gives the following comparison.

\begin{corollary}[The price of symmetry at quadratic dimension]
\label{cor:price-symmetry}
For $m\ge2$, set
\[
 N_m:=h_+\!\left(\left\lfloor\frac{m^2}{2}\right\rfloor\right).
\]
Then, as $m\to\infty$, $N_m=m^2/2+O(m^{1/3})$ and
\begin{equation}\label{eq:quadratic-unrestricted-exact}
 \frac{S_{m,N_m}}{N_m^{(m+1)/2}}\le1.
\end{equation}
On the other hand,
\begin{equation}\label{eq:quadratic-symmetric-lower}
 \frac{S_{m,N_m}^{\sym}}{N_m^{(m+1)/2}}
 \ge
 \left(\sqrt{\frac2e}+o(1)\right)\frac{\sqrt{m!}}m.
\end{equation}
Consequently,
\[
 \frac{S_{m,N_m}^{\sym}}{S_{m,N_m}}
 \ge
 \left(\sqrt{\frac2e}+o(1)\right)\frac{\sqrt{m!}}m.
\]
\end{corollary}

For the lower bound we restrict a symmetric form to the diagonal and use Parseval on the discrete cube.  The square-free Walsh coefficients then have modulus $m!$.

For the upper bound, fix cube vectors $x^{(1)},\ldots,x^{(m)}$ and let $R$ be the Gram matrix of their normalizations.  The $L^2$ norm, with respect to the orbit signs, is controlled by $\per R$, while
$\per R/m!$ is the squared Hilbert norm of the symmetrized tensor product.  A quantitative rigidity estimate shows that a large permanent forces the normalized vectors, after sign changes, to be close to one another.  The variables that are sufficiently close are recentered at the direction of their sum.  The resulting transverse components have vanishing sum before normalization, and Newton identities control the higher symmetric products.  The remaining estimates are obtained from Gaussian widths.

The proof of Theorem~\ref{thm:unrestricted-summary} is separate and follows from the rectangular Hadamard construction in Lemma~\ref{lem:block-ksz}.

\begingroup
\color{black}
\section{Proof of the lower bound in Theorem~A}\label{sec:walsh-lower}

The lower bound is deterministic.  Symmetry becomes rigid on the diagonal,
where every permutation orbit contributes coherently to a single Walsh
character.  Parseval then converts orbit multiplicity into norm.  This
mechanism already identifies the factorial scale and, after optimization in
the dimension, singles out $n\asymp m^2$.

Let $A_{m,n}:(\ell_\infty^n)^m\to\R$ be a real symmetric unimodular
$m$-linear form.  Define its diagonal polynomial
\[
 P_A:\R^n\longrightarrow\R,
 \qquad
 P_A(x):=A_{m,n}(x,\ldots,x).
\]
For each $E\subset\{1,\ldots,n\}$, define the Walsh character
\[
 \chi_E:\{-1,1\}^n\longrightarrow\{-1,1\},
 \qquad
 \chi_E(x):=\prod_{j\in E}x_j,
 \qquad
 \widehat P_A(E)
 :=
 2^{-n}\sum_{x\in\{-1,1\}^n}P_A(x)\chi_E(x).
\]

\begin{lemma}\label{lem:squarefree}
If $E=\{j_1,\ldots,j_m\}\subset\{1,\ldots,n\}$ has cardinality $m$, then
\begin{equation}\label{eq:squarefree-walsh}
 \widehat P_A(E)=m!\,\varepsilon_{j_1,\ldots,j_m}.
\end{equation}
In particular, $|\widehat P_A(E)|=m!$.
\end{lemma}

\begin{proof}
Expand
\[
 P_A(x)
 =
 \sum_{i_1,\ldots,i_m=1}^n
 \varepsilon_{i_1,\ldots,i_m}x_{i_1}\cdots x_{i_m}.
\]
By orthogonality of the Walsh characters, a monomial has nonzero inner
product with $\chi_E$ precisely when the indices occurring an odd number of
times form $E$.  Since the total degree and $|E|$ are both $m$, every element
of $E$ must occur exactly once and no other index can occur.  There are $m!$
possible orderings.  Symmetry makes their coefficients equal, proving
\eqref{eq:squarefree-walsh}.
\end{proof}

\begin{proposition}\label{prop:parseval}
For every $n\ge m$,
\begin{equation}\label{eq:parsevallower}
 S_{m,n}^{\sym}\ge m!\sqrt{\binom nm}.
\end{equation}
Consequently,
\begin{equation}\label{eq:loweralln}
 C_m^{\sym}
 \ge
 \sup_{n\ge m}
 \frac{m!\sqrt{\binom nm}}{n^{(m+1)/2}}.
\end{equation}
\end{proposition}

\begin{proof}
Parseval and Lemma~\ref{lem:squarefree} give
\begin{align*}
 2^{-n}\sum_{x\in\{-1,1\}^n}|P_A(x)|^2
 &=
 \sum_{E\subset\{1,\ldots,n\}}
 |\widehat P_A(E)|^2\\
 &\ge
 \sum_{|E|=m}|\widehat P_A(E)|^2
 =
 (m!)^2\binom nm.
\end{align*}
Hence
\[
 \|A_{m,n}\|
 \ge
 \max_{x\in\{-1,1\}^n}|P_A(x)|
 \ge
 m!\sqrt{\binom nm}.
\]
Minimizing over symmetric unimodular forms proves
\eqref{eq:parsevallower}; division by $n^{(m+1)/2}$ and the supremum over
$n$ give \eqref{eq:loweralln}.
\end{proof}

\begin{lemma}\label{lem:walsh-product-expansion}
Let $c>0$ be fixed and let $n=n_m$ be integers satisfying
$n_m\sim cm^2$.  Then
\begin{equation}\label{eq:walsh-product-expansion}
 \prod_{k=0}^{m-1}\left(1-\frac{k}{n_m}\right)
 =\exp\left\{-\frac1{2c}+o(1)\right\}.
\end{equation}
\end{lemma}

\begin{proof}
Since $n_m\sim cm^2$, there is $m_1$ such that, for $m\ge m_1$,
\[
 0\le \frac{k}{n_m}\le \frac{m-1}{n_m}\le \frac12
 \qquad(0\le k\le m-1).
\]
For $0\le t\le1/2$,
\[
 \log(1-t)=-t+R(t),
 \qquad |R(t)|\le C t^2
\]
with an absolute constant $C$.  Hence
\begin{align*}
 \sum_{k=0}^{m-1}\log\left(1-\frac{k}{n_m}\right)
 &=-\frac1{n_m}\sum_{k=0}^{m-1}k
   +O\left(\frac1{n_m^2}\sum_{k=0}^{m-1}k^2\right)\\
 &=-\frac{m(m-1)}{2n_m}
   +O\left(\frac{m(m-1)(2m-1)}{6n_m^2}\right).
\end{align*}
Now
\[
 \frac{m(m-1)}{2n_m}=\frac1{2c}+o(1)
\]
and, because $n_m\asymp m^2$,
\[
 \frac{m(m-1)(2m-1)}{6n_m^2}=O(m^{-1})=o(1).
\]
Therefore
\[
 \sum_{k=0}^{m-1}\log\left(1-\frac{k}{n_m}\right)
 =-\frac1{2c}+o(1),
\]
and exponentiation gives \eqref{eq:walsh-product-expansion}.
\end{proof}

\begin{lemma}\label{lem:walsh-c-optimization}
The function
\[
 \varphi:(0,\infty)\longrightarrow(0,\infty),
 \qquad
 \varphi(c):=c^{-1/2}e^{-1/(4c)},
\]
has a unique maximum at $c=1/2$, and
\[
 \varphi(1/2)=\sqrt{\frac2e}.
\]
\end{lemma}

\begin{proof}
Since $\varphi(c)>0$, it is enough to differentiate its logarithm:
\[
 \log\varphi(c)=-\frac12\log c-\frac1{4c}.
\]
Thus
\[
 \frac{d}{dc}\log\varphi(c)
 =-\frac1{2c}+\frac1{4c^2}
 =\frac{1-2c}{4c^2}.
\]
This derivative is positive for $0<c<1/2$, vanishes at $c=1/2$, and is
negative for $c>1/2$.  Hence $c=1/2$ is the unique maximizer.  Finally,
\[
 \varphi(1/2)
 =(1/2)^{-1/2}e^{-1/2}
 =\sqrt2\,e^{-1/2}
 =\sqrt{\frac2e}.
\]
\end{proof}

\begin{proposition}\label{prop:lower}
As $m\to\infty$,
\begin{equation}\label{eq:lower}
 C_m^{\sym}
 \ge
 \left(\sqrt{\frac2e}+o(1)\right)\frac{\sqrt{m!}}m.
\end{equation}
\end{proposition}

\begin{proof}
From Proposition~\ref{prop:parseval}, for every integer $n\ge m$,
\[
 C_m^{\sym}
 \ge \frac{m!}{n^{(m+1)/2}}\sqrt{\binom nm}.
\]
Using
\[
 \binom nm
 =\frac{n(n-1)\cdots(n-m+1)}{m!}
 =\frac{n^m}{m!}
   \prod_{k=0}^{m-1}\left(1-\frac{k}{n}\right),
\]
we obtain, without suppressing any factor,
\begin{align}
 \frac{m!}{n^{(m+1)/2}}\sqrt{\binom nm}
 &=\frac{m!}{n^{(m+1)/2}}
   \frac{n^{m/2}}{\sqrt{m!}}
   \left[
    \prod_{k=0}^{m-1}\left(1-\frac{k}{n}\right)
   \right]^{1/2}\notag\\
 &=\frac{\sqrt{m!}}{\sqrt n}
   \left[
    \prod_{k=0}^{m-1}\left(1-\frac{k}{n}\right)
   \right]^{1/2}.
 \label{eq:lowerproduct-open}
\end{align}
Fix $c>0$ and choose integers $n_m\sim cm^2$.  Then
\[
 \frac1{\sqrt{n_m}}
 =\frac1{m\sqrt c}\,(1+o(1)).
\]
By Lemma~\ref{lem:walsh-product-expansion},
\[
 \left[
  \prod_{k=0}^{m-1}\left(1-\frac{k}{n_m}\right)
 \right]^{1/2}
 =\exp\left\{-\frac1{4c}+o(1)\right\}
 =e^{-1/(4c)}(1+o(1)).
\]
Substitution into \eqref{eq:lowerproduct-open} yields
\begin{align*}
 C_m^{\sym}
 &\ge
 \sqrt{m!}\,
 \frac1{m\sqrt c}(1+o(1))
 e^{-1/(4c)}(1+o(1))\\
 &=\left(c^{-1/2}e^{-1/(4c)}+o(1)\right)
   \frac{\sqrt{m!}}m.
\end{align*}
Lemma~\ref{lem:walsh-c-optimization} shows that the largest constant produced
by this one-parameter choice is attained at $c=1/2$ and equals
$\sqrt{2/e}$.  Therefore
\[
 C_m^{\sym}
 \ge
 \left(\sqrt{\frac2e}+o(1)\right)\frac{\sqrt{m!}}m.
\]
The factor $m^{-1}$ is now explicit: it comes from
$\sqrt{n_m}^{-1}$ after the optimizing dimension $n_m\sim m^2/2$ is chosen.
\end{proof}

\endgroup

\section{Symmetrization and Newton identities}
\label{sec:symmetric-tensor-machinery}

The upper-bound argument uses the following tensor notation.  Let $H$ be a real or complex Hilbert space.  We distinguish the algebraic tensor product, on which the projective norm is defined, from its Hilbertian completion, on which symmetrization is an orthogonal projection.

For a real or complex Hilbert space $H$, let $H^{\otimes_{\rm alg} k}$ be the
$k$-fold algebraic tensor product and let $H^{\otimes_2 k}$ be its Hilbert
completion.  A permutation $\sigma\in S_k$ acts isometrically on both spaces
by permuting the tensor factors.  We write
\[
 P_{\sym,k}:=\frac1{k!}\sum_{\sigma\in S_k}U_\sigma.
\]
We also write $P_{\sym}^{(k)}:=P_{\sym,k}$ and suppress the degree when it
is clear.  At degree zero, $H^{\otimes0}=\operatorname{Sym}^0(H)$ is the
scalar field.
On $H^{\otimes_2 k}$ this is the orthogonal projection onto
\[
 \operatorname{Sym}_2^k(H):=P_{\sym,k}(H^{\otimes_2 k});
\]
on the algebraic tensor product its range will be denoted by
$\operatorname{Sym}_{\rm alg}^k(H)$.  In finite dimension the two spaces
coincide, and we simply write $\operatorname{Sym}^k(H)$.

If $u\in\operatorname{Sym}_{\rm alg}^j(H)$ and
$v\in\operatorname{Sym}_{\rm alg}^k(H)$, their symmetric product is
\begin{equation}\label{eq:symmetric-product-convention}
 u\,\widehat\otimes\,v:=P_{\sym,j+k}(u\otimes v).
\end{equation}
Thus $x^{\widehat\otimes k}=x^{\otimes k}$: no factorial is built into a
pure symmetric power.  The projective norm on $H^{\otimes_{\rm alg} k}$ is
\begin{equation}\label{eq:projective-norm-definition}
 \|u\|_\pi:=
 \inf\left\{
   \sum_\nu\prod_{j=1}^k\|x_{j,\nu}\|_2:
   u=\sum_\nu x_{1,\nu}\otimes\cdots\otimes x_{k,\nu}
 \right\}.
\end{equation}
Each $U_\sigma$ preserves the projective norm.  Hence $P_{\sym,k}$ is contractive for $\|\cdot\|_\pi$, and
\begin{equation}\label{eq:symmetric-projective-submultiplicative}
 \|u\widehat\otimes v\|_\pi\le \|u\|_\pi\|v\|_\pi.
\end{equation}
Every tensor to which $\|\cdot\|_\pi$ is applied below belongs to the algebraic tensor product.

For $v_1,\ldots,v_d\in H$, define
\begin{equation}\label{eq:elementary-power-sum-conventions}
 \begin{split}
 e_0(v)&:=1,\\
 e_\ell(v)&:=
 \sum_{1\le r_1<\cdots<r_\ell\le d}
 v_{r_1}\widehat\otimes\cdots\widehat\otimes v_{r_\ell}
 \quad (1\le\ell\le d),\\
 e_\ell(v)&:=0\quad(\ell>d),
 \qquad
 P_k(v):=\sum_{r=1}^d v_r^{\widehat\otimes k}.
 \end{split}
\end{equation}
The identity
\begin{equation}\label{eq:symmetric-generating-polynomial}
 \prod_{r=1}^d(1+t v_r)=\sum_{\ell\ge0}e_\ell(v)t^\ell
\end{equation}
holds in the symmetric tensor algebra.  Taking its formal logarithm and
then exponentiating gives
\begin{equation}\label{eq:newton-generating-identity-conventions}
 \sum_{\ell\ge0}e_\ell(v)t^\ell
 =\exp_{\widehat\otimes}\!\left(
   \sum_{k\ge1}\frac{(-1)^{k-1}}{k}P_k(v)t^k
 \right).
\end{equation}
Here $\exp_{\widehat\otimes}(w):=\sum_{j\ge0}
w^{\widehat\otimes j}/j!$.  Equivalently,
\begin{equation}\label{eq:newton-recursion-conventions}
 \ell e_\ell(v)=
 \sum_{k=1}^\ell(-1)^{k-1}
 e_{\ell-k}(v)\widehat\otimes P_k(v).
\end{equation}
Taking projective norms in this identity and using \eqref{eq:symmetric-projective-submultiplicative} yields
\begin{equation}\label{eq:newton-norm-majorant-conventions}
 \sum_{\ell\ge0}\|e_\ell(v)\|_\pi |t|^\ell
 \le
 \exp\left(\sum_{k\ge1}\frac{\|P_k(v)\|_\pi}{k}|t|^k\right)
\end{equation}
whenever the series on the right converges.

We take complex inner products to be linear in the first variable.  For unit
vectors $y_1,\ldots,y_m\in H$, let
$R=(\langle y_r,y_s\rangle)_{r,s=1}^m$ be their Gram matrix.  Since $P_{\sym,m}$ is an orthogonal projection,
\begin{align}
 \bigl\|P_{\sym,m}(y_1\otimes\cdots\otimes y_m)\bigr\|_2^2
 &=\bigl\langle P_{\sym,m}(y_1\otimes\cdots\otimes y_m),
                    y_1\otimes\cdots\otimes y_m\bigr\rangle \notag\\
 &=\frac1{m!}\sum_{\sigma\in S_m}
   \prod_{r=1}^m\langle y_{\sigma(r)},y_r\rangle
 =\frac{\operatorname{per}R}{m!}.
 \label{eq:abstract-permanent-identity}
\end{align}
The quantity in \eqref{eq:abstract-permanent-identity} is positive.  Indeed, each $y_r^\perp$ is a proper subspace of $H$, and a vector space over $\mathbb R$ or $\mathbb C$ cannot be covered by finitely many proper subspaces.  Hence we may choose $h\in H\setminus\bigcup_{r=1}^m y_r^\perp$.  Since $h^{\otimes m}$ is symmetric,
\[
 \left\langle
 P_{\sym,m}(y_1\otimes\cdots\otimes y_m),h^{\otimes m}
 \right\rangle
 =\prod_{r=1}^m\langle y_r,h\rangle\ne0.
\]
It is at most $1$, because $P_{\sym,m}$ is an orthogonal projection and $\|y_1\otimes\cdots\otimes y_m\|_2=1$.

\section{Permanent rigidity}
\label{sec:permanent-rigidity}

The next estimate gives a dimension-free consequence of a large Gram permanent.

\begin{theorem}
\label{thm:abstract-permanent-rigidity}
Let $y_1,\ldots,y_m$ be unit vectors in a real or complex Hilbert space and
let $R$ be their Gram matrix.  Set
\[
 \alpha:=\frac{\operatorname{per}R}{m!},
 \qquad L:=-\log\alpha.
\]
Then $0<\alpha\le1$, and there are an index $r_0$ and phases
$\omega_1,\ldots,\omega_m\in\mathbb T$ (signs in the real case), with
$\omega_{r_0}=1$, such that
\begin{equation}\label{eq:abstract-rigidity}
 \sum_{r=1}^m\|\omega_r y_r-y_{r_0}\|_2^2
 \le 2m\bigl(1-e^{-2L/m}\bigr)
 \le4L.
\end{equation}
\end{theorem}

\begin{proof}
Write $R=(\rho_{rs})$ and
$s_j:=\sum_{i=1}^m|\rho_{ij}|^2$.  The permanent inequality of Carlen,
Lieb and Loss \cite[Theorem~1.1, Eq.~(1.1)]{CarlenLiebLoss} gives
\[
 e^{-L}=\alpha\le m^{-m/2}\prod_{j=1}^m s_j^{1/2}.
\]
Squaring and rearranging,
\[
 \prod_{j=1}^m s_j\ge m^m e^{-2L}.
\]
Hence
\[
 \left(\prod_{j=1}^m s_j\right)^{1/m}
 \ge m e^{-2L/m}.
\]
The arithmetic--geometric mean inequality therefore yields
\[
 \frac1m\sum_{j=1}^m s_j
 \ge \left(\prod_{j=1}^m s_j\right)^{1/m}
 \ge m e^{-2L/m},
\]
and thus
\begin{equation}\label{eq:rigidity-sum-sj}
 \sum_{j=1}^m s_j\ge m^2e^{-2L/m}.
\end{equation}
Since $R$ has diagonal one,
\[
 \sum_{r<s}(1-|\rho_{rs}|^2)
 =\frac12\left(m^2-\sum_{j=1}^m s_j\right)
 \le\frac{m^2}{2}\bigl(1-e^{-2L/m}\bigr).
\]
Consequently the row defects
\[
 D_r:=\sum_{s\ne r}(1-|\rho_{rs}|^2)
\]
have average at most $m(1-e^{-2L/m})$.
\begingroup
Choose $r_0$ so that
\[
 D_{r_0}\le m\bigl(1-e^{-2L/m}\bigr).
\]
For $1\le r\le m$, define
\[
 \omega_r:=
 \begin{cases}
  \overline{\rho_{rr_0}}/|\rho_{rr_0}|,&\rho_{rr_0}\ne0,\\[4pt]
  1,&\rho_{rr_0}=0.
 \end{cases}
\]
Then $\omega_{r_0}=1$ and, since the inner product is linear in its first
variable,
\[
 \langle\omega_r y_r,y_{r_0}\rangle
 =\omega_r\rho_{rr_0}=|\rho_{rr_0}|.
\]
In the real case these phases are signs.
\endgroup
Then
\begin{align*}
 \sum_{r=1}^m\|\omega_r y_r-y_{r_0}\|_2^2
 &=2\sum_{r\ne r_0}(1-|\rho_{rr_0}|)\\
 &\le2\sum_{r\ne r_0}(1-|\rho_{rr_0}|^2)\\
 &=2D_{r_0}
 \le2m\bigl(1-e^{-2L/m}\bigr).
\end{align*}
Finally, $1-e^{-u}\le u$ gives the bound $4L$.
\end{proof}

Thus, if the symmetric projection has Hilbert norm close to $1$, then after multiplying the factors by suitable phases, all of them are close to one fixed factor.  The estimate is independent of the dimension of $H$.

\section{Barycentric recentering and transverse cancellation}
\label{sec:barycentric-recentering}

Assume that the vectors lie in a sufficiently small spherical cap centered at $e_0$.  We replace $e_0$ by the normalized sum of the vectors.  With this choice, the components orthogonal to the new direction have sum zero.  The following theorem records the estimates needed later.

\begingroup
\begin{theorem}
\label{thm:abstract-moving-anchor}
There is an absolute $\eta_0\in(0,1/4]$ with the following property.  Let
$H$ be a real Hilbert space, let $d\ge1$, $S\ge1$, $0\le L\le S$, and
$0<\eta\le\eta_0$.  Suppose that $e_0,y_1,\ldots,y_d$ are unit vectors and
\begin{equation}\label{eq:barycentric-hypotheses}
 \sum_{r=1}^d\|y_r-e_0\|_2^2\le4L,
 \qquad
 \max_r\|y_r-e_0\|_2\le\frac{\eta}{S}.
\end{equation}
Then the following assertions hold:
\begin{enumerate}[label=\textup{(\roman*)}]
\item The vector
\[
 Y:=\sum_{r=1}^d y_r
\]
is nonzero.  Moreover, if
\[
 e:=\frac{Y}{\|Y\|_2},
\]
then
\begin{align}
 \|e-e_0\|_2&\le\frac{\eta}{S},
 \label{eq:abstract-moving-anchor-distance}\\
 \sum_{r=1}^d\|y_r-e\|_2^2&\le4L,
 \label{eq:abstract-moving-energy}\\
 \max_r\|y_r-e\|_2&\le\frac{2\eta}{S}.
 \label{eq:abstract-moving-individual}
\end{align}

\item Set
\[
 a_r:=\langle y_r,e\rangle,
 \qquad z_r:=y_r-a_re.
\]
Then
\begin{equation}\label{eq:abstract-longitudinal-positive}
 a_r\ge\frac12.
\end{equation}
Consequently, $v_r:=a_r^{-1}z_r$ is well defined, and
\begin{equation}\label{eq:barycentric-decomposition}
 y_r=a_re+z_r=a_r(e+v_r),
 \qquad z_r\perp e.
\end{equation}
Furthermore,
\begin{align}
 \sum_{r=1}^d z_r&=0,
 \label{eq:abstract-z-cancellation}\\
 \delta:=\max_r\|v_r\|_2&\le\frac{4\eta}{S},
 \label{eq:abstract-delta}\\
 V:=\sum_{r=1}^d\|v_r\|_2^2&\le16L\le16S,
 \label{eq:abstract-V}\\
 \left\|\sum_{r=1}^d v_r\right\|_2
 &\le\frac{16\eta L}{S}\le16\eta.
 \label{eq:abstract-P1}
\end{align}

\item The following exact expansion holds:
\begin{equation}\label{eq:barycentric-exact-expansion}
 P_{\sym,d}(y_1\otimes\cdots\otimes y_d)
 =\left(\prod_{r=1}^d a_r\right)
   \sum_{\ell=0}^d
   e_\ell(v)\widehat\otimes e^{\otimes(d-\ell)}.
\end{equation}
\end{enumerate}
\end{theorem}

\begin{proof}
$(i)$ Set $\varepsilon:=\eta/S$ and $\theta_0:=2\arcsin(\varepsilon/2)$.  Since $\eta_0\le1/4$, we have $\theta_0<\pi/2$.  For each $r$, let $\theta_r\in[0,\pi]$ be the angle between $y_r$ and $e_0$.  Since
\[
 \|y_r-e_0\|_2
 =2\sin\frac{\theta_r}{2}
 \le\varepsilon,
\]
we have $0\le\theta_r\le\theta_0$.  If $\theta_r>0$, define
\[
 u_r:=
 \frac{y_r-\langle y_r,e_0\rangle e_0}
 {\|y_r-\langle y_r,e_0\rangle e_0\|_2};
\]
if $\theta_r=0$, set $u_r:=0$.  Then $u_r\perp e_0$, $\|u_r\|_2=1$ whenever $\theta_r>0$, and
\[
 y_r=\cos\theta_r\,e_0+\sin\theta_r\,u_r.
\]
Consequently, $Y=Ae_0+W$, where
\[
 A:=\sum_r\cos\theta_r>0,
 \qquad
 \|W\|_2\le\sum_r\sin\theta_r
       \le A\tan\theta_0.
\]
Since $A>0$ and $\|W\|_2\le A\tan\theta_0$, we have $Y\ne0$.  If $\theta$ is the angle between $Y$ and $e_0$, then $\theta\le\theta_0$, so
\[
 \|e-e_0\|_2
 =2\sin\frac{\theta}{2}
 \le2\sin\frac{\theta_0}{2}
 =\varepsilon.
\]
This proves \eqref{eq:abstract-moving-anchor-distance}.

For any unit vector $u$,
\[
 \sum_{r=1}^d\|y_r-u\|_2^2=2d-2\langle Y,u\rangle.
\]
Because $\langle Y,u\rangle\le\|Y\|_2=\langle Y,e\rangle$ for every unit vector $u$, the left-hand side is minimized at $u=e$.  Taking $u=e_0$ and using \eqref{eq:barycentric-hypotheses} gives \eqref{eq:abstract-moving-energy}.  Together with \eqref{eq:abstract-moving-anchor-distance} and the cap condition, this gives \eqref{eq:abstract-moving-individual}, and hence proves \textup{(i)}.

\smallskip

$(ii)$ Since $Y=\|Y\|_2e$, projection onto $e^\perp$ gives $\sum_r z_r=0$.  Also,
\[
 a_r=\langle y_r,e\rangle
 =1-\frac12\|y_r-e\|_2^2\ge\frac12,
\]
after fixing $\eta_0\le1/4$.  Since $\|z_r\|_2\le\|y_r-e\|_2$,
\[
 \max_r\|v_r\|_2\le2\max_r\|z_r\|_2\le\frac{4\eta}{S}
\]
and
\[
 \sum_r\|v_r\|_2^2
 \le4\sum_r\|z_r\|_2^2
 \le4\sum_r\|y_r-e\|_2^2
 \le16L.
\]
From $\|y_r\|_2=1$ and $y_r=a_re+z_r$ we have $a_r=(1-\|z_r\|_2^2)^{1/2}$.  Since $\|z_r\|_2\le1/2$,
\[
 0\le a_r^{-1}-1\le2\|z_r\|_2^2.
\]
Therefore, using $\sum_r z_r=0$,
\begin{align*}
 \left\|\sum_rv_r\right\|_2
 &=\left\|\sum_r(a_r^{-1}-1)z_r\right\|_2\\
 &\le2\max_r\|z_r\|_2\sum_r\|z_r\|_2^2
 \le\frac{16\eta L}{S}.
\end{align*}
This proves \textup{(ii)}.

\smallskip

$(iii)$ By multilinearity,
\[
 P_{\sym,d}\!\left(\bigotimes_{r=1}^d(e+v_r)\right)
 =\sum_{I\subset\{1,\ldots,d\}}
 P_{\sym,d}\!\left(
   \bigotimes_{r\in I}v_r\otimes e^{\otimes(d-|I|)}
 \right),
\]
where the factors indexed by $I$ are taken in increasing order.  Since
$P_{\sym,d}(P_{\sym,|I|}\otimes\mathrm{Id})=P_{\sym,d}$, the sum of the terms with $|I|=\ell$ is
$e_\ell(v)\widehat\otimes e^{\otimes(d-\ell)}$.  Multiplication by $\prod_r a_r$ proves \eqref{eq:barycentric-exact-expansion}, completing the proof of \textup{(iii)}.
\end{proof}
\endgroup

The vectors $v_r$ need not have zero sum, but \eqref{eq:abstract-P1} gives a uniform bound for $P_1(v)=\sum_r v_r$.  For $k\ge2$, the quantities $P_k(v)$ are controlled by $\delta$ and $V$.  Newton's identities then give bounds for all $e_\ell(v)$.  The factor $\prod_r a_r$ supplies the compensating quadratic term.  Since $\|y_r\|_2=1$ and $y_r=a_r(e+v_r)$,
\[
 a_r=(1+\|v_r\|_2^2)^{-1/2}.
\]
\begingroup
\color{black}
Thus the recentering supplies simultaneously the three scales used in the
Newton estimate:
\[
 V=O(S),\qquad \delta=O(\eta/S),\qquad
 \|P_1(v)\|_\pi=O(\eta).
\]
The first two control the higher power sums, while the last controls the
linear Newton term.
\endgroup

\begin{proposition}
\label{prop:abstract-moving-estimates}
There are absolute constants $c,C>0$ such that the following holds, after the constant $\eta_0$ in Theorem~\ref{thm:abstract-moving-anchor}
has been chosen sufficiently small.  Let $S\ge1$,
$0<\eta\le\eta_0$, and let $v_1,\ldots,v_d$ belong to a real Hilbert space.
Suppose that
\begin{equation}\label{eq:compensated-newton-hypotheses}
 \delta:=\max_r\|v_r\|_2\le\frac{4\eta}{S},
 \qquad
 V:=\sum_r\|v_r\|_2^2\le16S,
 \qquad
 \left\|\sum_rv_r\right\|_2\le16\eta.
\end{equation}
Define
\[
 a_r:=(1+\|v_r\|_2^2)^{-1/2},
 \qquad A(v):=\prod_{r=1}^d a_r,
 \qquad t_S:=1+\frac cS.
\]
Then
\begin{align}
 A(v)\sum_{\ell\ge0}\|e_\ell(v)\|_\pi t_S^\ell&\le C,
 \label{eq:abstract-moving-mgf}\\
 A(v)\sum_{\ell\ge1}\sqrt{\log(e\ell)}\,
          \|e_\ell(v)\|_\pi
 &\le C\sqrt{\log(e+S)}.
 \label{eq:abstract-moving-weighted}
\end{align}
\end{proposition}

\begin{proof}
For $k\ge2$,
\[
 \|P_k(v)\|_\pi
 \le\sum_r\|v_r\|_2^k
 \le V\delta^{k-2},
\]
and $\|P_1(v)\|_\pi\le16\eta$.  Substitution in \eqref{eq:newton-norm-majorant-conventions} gives, for $\delta t<1$,
\begin{equation}\label{eq:compensated-newton-majorant-proof}
 \sum_{\ell\ge0}\|e_\ell(v)\|_\pi t^\ell
 \le\exp\left\{
   16\eta t+\frac{Vt^2}{2}
   +\frac{V\delta t^3}{3(1-\delta t)}
 \right\}.
\end{equation}
Also, $\log(1+s)\ge s-s^2/2$ for $s\ge0$, and hence
\begin{equation}\label{eq:longitudinal-compensation-proof}
 \log A(v)
 =-\frac12\sum_r\log(1+\|v_r\|_2^2)
 \le-\frac V2+\frac{\delta^2V}{4}.
\end{equation}
Choose $0<c\le1$, set $t=t_S$, and take $\eta_0$ small enough that
\[
 \delta t_S\le4\eta_0(1+c)\le\frac12.
\]
Combining \eqref{eq:compensated-newton-majorant-proof} and \eqref{eq:longitudinal-compensation-proof},
\begin{align}
 \log\!\left[
 A(v)\sum_{\ell\ge0}\|e_\ell(v)\|_\pi t_S^\ell
 \right]
 &\le 16\eta t_S+\frac V2(t_S^2-1)
 +\frac{V\delta t_S^3}{3(1-\delta t_S)}
 +\frac{\delta^2V}{4}.
 \label{eq:compensated-newton-combined}
\end{align}
Because $t_S=1+c/S$, $V\le16S$, and $S\ge1$,
\begin{align*}
 V(t_S^2-1)
 &=V\left(\frac{2c}{S}+\frac{c^2}{S^2}\right)
 \le32c+16c^2,\\
 V\delta
 &\le16S\frac{4\eta}{S}=64\eta,\\
 \delta^2V
 &\le\left(\frac{4\eta}{S}\right)^2 16S
 =\frac{256\eta^2}{S}\le256\eta^2.
\end{align*}
Since $\delta t_S\le1/2$ and $t_S\le2$, the right-hand side of \eqref{eq:compensated-newton-combined} is bounded by an absolute constant.  This proves \eqref{eq:abstract-moving-mgf}.

To prove \eqref{eq:abstract-moving-weighted}, note first that, after decreasing $c$ if necessary, $\log t_S\ge c/(2S)$.  We claim that
\[
 \sqrt{\log(e\ell)}
 \le C\sqrt{\log(e+S)}\,t_S^\ell,
 \qquad \ell\ge1,
\]
for all $S\ge1$ and $\ell\ge1$.  If $1\le\ell\le S$, this follows from $\log(e\ell)\le1+\log S\le2\log(e+S)$.  If $\ell>S$, set $x=\ell/S>1$.  Then $t_S^\ell\ge e^{cx/2}$ and
\[
 \sqrt{\log(e\ell)}
 \le \sqrt{\log(e+S)}+\sqrt{\log(ex)}.
\]
Since $\log(e+S)\ge1$ and
$\sup_{x\ge1}(1+\sqrt{\log(ex)})e^{-cx/2}<\infty$, the claimed bound follows.  Multiplying it by $A(v)\|e_\ell(v)\|_\pi$ and summing over $\ell$ gives \eqref{eq:abstract-moving-weighted} from \eqref{eq:abstract-moving-mgf}.
\end{proof}

Combining Theorem~\ref{thm:abstract-moving-anchor} with Proposition~\ref{prop:abstract-moving-estimates} gives uniform projective-norm bounds for the terms in \eqref{eq:barycentric-exact-expansion}.

\begin{remark}[Relation with norm-attainment rigidity]
For nonzero symmetric multilinear forms of degree at least three on a complex Hilbert space, Carando and Rodr\'iguez proved that every norm-attaining tuple is collinear and obtained a qualitative Bollob\'as-type stability result; see \cite[Theorem~1.1 and Corollary~1.2]{CarandoRodriguez}.  Theorem~\ref{thm:abstract-permanent-rigidity} gives the quantitative estimate used here.
\end{remark}

\section{The symmetric orbit family and permanent localization}\label{sec:permanent-geometry}

For a finite set $I$, let
\begin{equation}\label{eq:counting-measure}
 \mu_I:=2^{-|I|}\sum_{\varepsilon\in\{-1,1\}^{I}}\delta_\varepsilon
\end{equation}
be the normalized counting measure on $\{-1,1\}^I$.  If $H$ is a finite-dimensional real Hilbert space, let $\gamma_H$ be the centered Gaussian probability measure on $H$ with identity covariance.  We write $\mu$ and $\gamma$ when the underlying space is clear.

For the symmetric random form, one independent sign is assigned to each permutation orbit.  The $L^2(\mu)$ norm at a fixed configuration is therefore the Euclidean norm of the vector of orbit sums.  Lemma~\ref{lem:permanent-control} bounds this norm by the Gram permanent.  The same permanent is related to the symmetrized tensor product by \eqref{eq:abstract-permanent-identity}.

Let $\mathcal O_{m,n}$ be the set of orbits of
$\{1,\ldots,n\}^m$ under permutations of the $m$ positions.  For
$\varepsilon=(\varepsilon_{\mathcal O})_{\mathcal O\in\mathcal O_{m,n}}
\in\{-1,1\}^{\mathcal O_{m,n}}$, define
\[
 A_{m,n}^{\varepsilon}:(\ell_\infty^n)^m\longrightarrow\R
\]
by
\begin{equation}\label{eq:random-symmetric-form}
 A_{m,n}^{\varepsilon}(x^{(1)},\ldots,x^{(m)})
 :=
 \sum_{\mathcal O\in\mathcal O_{m,n}}
 \varepsilon_{\mathcal O}c_{\mathcal O}(x),
\end{equation}
where
\begin{equation}\label{eq:orbit-coefficient}
 c_{\mathcal O}(x)
 :=
 \sum_{(i_1,\ldots,i_m)\in\mathcal O}
 x^{(1)}_{i_1}\cdots x^{(m)}_{i_m}.
\end{equation}
We write
\[
 c(x):=\bigl(c_{\mathcal O}(x)\bigr)_{\mathcal O\in\mathcal O_{m,n}},
\]
so that
\[
 \int |A_{m,n}^{\varepsilon}(x)|^2\,d\mu(\varepsilon)=\|c(x)\|_2^2.
\]
For real multilinear forms on $\ell_\infty^n$, the norm is attained on
cube vertices.  Define the configuration space
\[
 \mathfrak X_{m,n}:=(\{-1,1\}^n)^m.
\]
For $x=(x^{(1)},\ldots,x^{(m)})\in\mathfrak X_{m,n}$, define
\[
 y_r:=n^{-1/2}x^{(r)}\in S^{n-1},\qquad 1\le r\le m.
\]
The Gram map and the normalized symmetrized-product map are
\begin{align}
 R:\mathfrak X_{m,n}&\longrightarrow\R^{m\times m},
 &R(x)&:=\bigl(\langle y_r,y_s\rangle\bigr)_{r,s=1}^m,\label{eq:R-definition}\\
 \Phi:\mathfrak X_{m,n}&\longrightarrow\operatorname{Sym}^m(\R^n),
 &\Phi(x)&:=P_{\sym}(y_1\otimes\cdots\otimes y_m).\label{eq:Phi-definition}
\end{align}
We also define the unnormalized symmetrized-product map
\begin{equation}\label{eq:Psi-definition}
 \Psi:\mathfrak X_{m,n}\longrightarrow\operatorname{Sym}^m(\R^n),
 \qquad
 \Psi(x):=P_{\sym}(x^{(1)}\otimes\cdots\otimes x^{(m)})
 =n^{m/2}\Phi(x).
\end{equation}
Here $P_{\sym}:(\R^n)^{\otimes_2 m}\to\operatorname{Sym}^m(\R^n)$ is
the orthogonal symmetrization defined in Section~\ref{sec:symmetric-tensor-machinery}.

\subsection{Orbit coefficients and the permanent statistic}

\begin{lemma}
\label{lem:permanent-control}
Let
\[
 G(x):=
 \bigl(\langle x^{(r)},x^{(s)}\rangle\bigr)_{r,s=1}^m
 =nR(x).
\]
Then
\begin{equation}\label{eq:permanent-identity}
 \|\Phi(x)\|_2^2
 =
 \frac{\per R(x)}{m!},
 \qquad
 \|\Psi(x)\|_2^2
 =
 \frac{\per G(x)}{m!}.
\end{equation}
Moreover,
\begin{equation}\label{eq:percontrol}
 \sum_{\mathcal O\in\mathcal O_{m,n}}
 |c_{\mathcal O}(x)|^2
 \le
 \per G(x)
 =
 n^m\per R(x).
\end{equation}
More generally, write
\[
 v=\sum_{i_1,\ldots,i_m=1}^n
 v_{i_1,\ldots,i_m}\,
 e_{i_1}\otimes\cdots\otimes e_{i_m}
\]
and define
\[
 c_{\mathcal O}(v):=
 \sum_{(i_1,\ldots,i_m)\in\mathcal O}v_{i_1,\ldots,i_m},
 \qquad
 c(v):=(c_{\mathcal O}(v))_{\mathcal O\in\mathcal O_{m,n}}.
\]
Then
\begin{equation}\label{eq:increment-projection}
 \sum_{\mathcal O}|c_{\mathcal O}(v)|^2
 \le m!\,\|P_{\sym}v\|_2^2.
\end{equation}
\end{lemma}

\begin{proof}
First take
\[
 u:=x^{(1)}\otimes\cdots\otimes x^{(m)}.
\]
If an orbit $\mathcal O$ has cardinality $d_{\mathcal O}$, then every
coefficient of $P_{\sym}u$ on that orbit equals
$c_{\mathcal O}(x)/d_{\mathcal O}$.  Hence
\begin{equation}\label{eq:orbit-projection-coefficients}
 \|P_{\sym}u\|_2^2
 =
 \sum_{\mathcal O}
 \frac{|c_{\mathcal O}(x)|^2}{d_{\mathcal O}}
 \ge
 \frac1{m!}\sum_{\mathcal O}|c_{\mathcal O}(x)|^2.
\end{equation}
On the other hand,
\begin{align*}
 \|P_{\sym}u\|_2^2
 &=
 \langle P_{\sym}u,u\rangle\\
 &=
 \frac1{m!}\sum_{\sigma\in S_m}
 \prod_{r=1}^m
 \langle x^{(r)},x^{(\sigma(r))}\rangle
 =
 \frac1{m!}\per G(x).
\end{align*}
The identity $\|P_{\sym}u\|_2^2=\per G(x)/m!$ is the formula for
$\|\Psi(x)\|_2^2$ in \eqref{eq:permanent-identity}; together with
\eqref{eq:orbit-projection-coefficients}, it also proves
\eqref{eq:percontrol}.  Applying the same computation to the normalized product
$y_1\otimes\cdots\otimes y_m$ gives the formula for
$\|\Phi(x)\|_2^2$.

For an arbitrary $v\in(\mathbb R^n)^{\otimes_2 m}$, the coefficient of $P_{\sym}v$ on an orbit
of cardinality $d_{\mathcal O}$ is again
$c_{\mathcal O}(v)/d_{\mathcal O}$.  Thus
\[
 \|P_{\sym}v\|_2^2
 =
 \sum_{\mathcal O}
 \frac{|c_{\mathcal O}(v)|^2}{d_{\mathcal O}}
 \ge
 \frac1{m!}\sum_{\mathcal O}|c_{\mathcal O}(v)|^2,
\]
which is \eqref{eq:increment-projection}.
\end{proof}

Define
\begin{equation}\label{eq:alpha}
 \alpha:\mathfrak X_{m,n}\longrightarrow(0,1],
 \qquad
 \alpha(x):=
 \frac{\per R(x)}{m!}
 =
 \|\Phi(x)\|_2^2.
\end{equation}
The positivity argument following \eqref{eq:abstract-permanent-identity}
shows that $\alpha(x)>0$, while orthogonal projection gives
$\alpha(x)\le1$.  Define the permanent loss by
\begin{equation}\label{eq:permanent-loss}
 L:\mathfrak X_{m,n}\longrightarrow[0,+\infty),
 \qquad
 L(x):=-\log\alpha(x).
\end{equation}
Thus Lemma~\ref{lem:permanent-control} gives the explicit $L^2(\mu)$ bound
\[
 \left(\int |A_{m,n}^{\varepsilon}(x)|^2\,d\mu(\varepsilon)\right)^{1/2}
 \le \sqrt{m!}\,n^{m/2}e^{-L(x)/2}.
\]
This is the gain used in the permanent tail.

For $u,v\in\{-1,1\}^n$, define the Hamming distance by
\[
 \distH(u,v):=\bigl|\{i\in\{1,\ldots,n\}:u_i\ne v_i\}\bigr|.
\]
The next lemma turns a large permanent into Hamming rigidity.  If each
$x^{(j)}$ is replaced by $\delta_jx^{(j)}$, every orbit coefficient is
multiplied by the same number $\prod_j\delta_j$.  The absolute value of the
orbit sum is therefore unchanged, and the natural localization is around a
signed anchor.

\begin{lemma}
\label{lem:hamming}
For $x\in\mathfrak X_{m,n}$, set $L=L(x)$.  Then there are
$r\in\{1,\ldots,m\}$ and signs $\delta_j\in\{-1,1\}$, with
$\delta_r=1$, such that
\begin{equation}\label{eq:hamming}
 \sum_{j\ne r}
 \distH\bigl(x^{(j)},\delta_jx^{(r)}\bigr)
 \le nL.
\end{equation}
\end{lemma}

\begin{proof}
Apply Theorem~\ref{thm:abstract-permanent-rigidity} to
$y_j=n^{-1/2}x^{(j)}$.  It supplies an index $r$ and signs $\delta_j$,
with $\delta_r=1$, for which
\[
 \sum_{j=1}^m\|y_j-\delta_jy_r\|_2^2\le4L.
\]
Since
\[
 \|y_j-\delta_jy_r\|_2^2
 =\frac4n\distH(x^{(j)},\delta_jx^{(r)}),
\]
the estimate furnished by Theorem~\ref{thm:abstract-permanent-rigidity} is exactly
\eqref{eq:hamming} after multiplication by $n/4$.
\end{proof}

\section{Entropy of permanent strata}

For each integer $k\ge0$, let
\begin{equation}\label{eq:permanent-stratum}
 \mathcal X_k
 :=
 \left\{
 x\in(\{-1,1\}^n)^m:
 e^{-(k+1)}<\alpha(x)\le e^{-k}
 \right\}.
\end{equation}
We use the binary entropy function
\[
 h:[0,1]\longrightarrow\R,
 \qquad
 h(u):=-u\log u-(1-u)\log(1-u),
\]
with $0\log0=0$.

\begin{lemma}\label{lem:hamming-ball}
Let $N\ge1$ and $0\le q\le N/2$ be integers.  Then
\begin{equation}\label{eq:hamming-ball}
 \sum_{j=0}^{q}\binom Nj
 \le
 \exp\left\{Nh\left(\frac qN\right)\right\}.
\end{equation}
\end{lemma}

\begin{proof}
If $q=0$, both sides of \eqref{eq:hamming-ball} are equal to $1$.  If
$q=N/2$, the left-hand side is at most $2^N$, while
$h(1/2)=\log 2$, so the same inequality holds.  It remains to consider
$0<q/N<1/2$.  Put
\[
 u:=\frac qN,
 \qquad
 s:=\frac{u}{1-u}\in(0,1).
\]
Since $s^j\ge s^q$ for $j\le q$,
\[
 s^q\sum_{j=0}^{q}\binom Nj
 \le
 \sum_{j=0}^{N}\binom Nj s^j
 =
 (1+s)^N.
\]
Substitution of $s=u/(1-u)$ gives \eqref{eq:hamming-ball}.
\end{proof}

\begin{lemma}\label{lem:count}
If $k+1\le(m-1)/2$, then
\begin{equation}\label{eq:count}
 |\mathcal X_k|
 \le
 m\,2^{n+m-1}
 \exp\left\{
 n(m-1)
 h\left(\frac{k+1}{m-1}\right)
 \right\}.
\end{equation}
For every $k$,
\begin{equation}\label{eq:trivial-stratum-count}
 |\mathcal X_k|\le2^{mn}.
\end{equation}
\end{lemma}

\begin{proof}
For $x\in\mathcal X_k$, Lemma~\ref{lem:hamming} supplies an anchor index,
an anchor vector, and $m-1$ signs such that the total number of
disagreements is at most $n(k+1)$.  There are at most
$m2^n2^{m-1}$ choices for these data.  Once they are fixed, the
disagreement pattern is a subset of a set with $n(m-1)$ elements and
cardinality at most $n(k+1)$.  Because the coordinates take only the values
$-1$ and $1$, this pattern determines all the remaining vectors.  Hence
\[
 |\mathcal X_k|
 \le
 m2^{n+m-1}
 \sum_{q\le n(k+1)}
 \binom{n(m-1)}q.
\]
When $(k+1)/(m-1)\le1/2$, Lemma~\ref{lem:hamming-ball} gives
\eqref{eq:count}.  For \eqref{eq:trivial-stratum-count}, no localization is used: the full
product cube $(\{-1,1\}^n)^m$ has exactly $2^{mn}$ elements, and
$\mathcal X_k$ is a subset of it.
\end{proof}

We repeatedly use
\begin{equation}\label{eq:entropy-simple}
 h(u)\le u\log\frac eu,
 \qquad 0<u\le\frac12.
\end{equation}

\section{Gaussian-width estimates and the permanent tail}
\label{sec:technical-toolkit}

We collect the integral estimates needed in the upper bound.  For a
nonempty bounded subset $S$ of a finite-dimensional real Hilbert space, its
absolute Gaussian width is
\begin{equation}\label{eq:gaussian-width-definition}
 w(S):=\int_{\operatorname{span}S} \sup_{s\in S}|\langle g,s\rangle|\,d\gamma_{\operatorname{span}S}(g).
\end{equation}
Equivalently, the integration is against the standard Gaussian measure on
$\operatorname{span}S$.
For a tensor $u\in H_1\otimes\cdots\otimes H_\ell$, define the injective norm
by
\begin{equation}\label{eq:injective-norm}
 \|u\|_\varepsilon
 :=
 \sup_{\|x_j\|_2\le1}
 |\langle u,x_1\otimes\cdots\otimes x_\ell\rangle|.
\end{equation}
The projective norm $\|\cdot\|_\pi$ was defined in
\eqref{eq:projective-norm-definition}.
\begingroup
The injective and projective norms satisfy the following dual estimate directly from their definitions.  Indeed, if
$v=\sum_\nu x_{1,\nu}\otimes\cdots\otimes x_{\ell,\nu}$, then
\[
 |\langle u,v\rangle|
 \le \|u\|_\varepsilon
     \sum_\nu\prod_{j=1}^\ell\|x_{j,\nu}\|_2.
\]
Taking the infimum over all such representations of $v$ gives
\[
 |\langle u,v\rangle|
 \le \|u\|_\varepsilon\|v\|_\pi.
\]
\endgroup
We use the same notation for their restrictions to symmetric tensor powers.

\begin{lemma}\label{lem:subgaussian-max}
Let $(\Omega,\nu)$ be a measure space with $\nu(\Omega)=1$ and let $Z_1,\ldots,Z_M:\Omega\to\mathbb R$ be measurable functions such that, for every
$\lambda\in\mathbb R$,
\[
 \int e^{\lambda Z_j}\,d\nu\le e^{\lambda^2\sigma^2/2}
 \qquad(1\le j\le M).
\]
Then
\[
 \int \max_{1\le j\le M}|Z_j|\,d\nu
 \le\sigma\sqrt{2\log(2M)}.
\]
\end{lemma}

\begin{proof}
For $\lambda>0$, Jensen's inequality and the pointwise bound
$e^{\lambda\max_j|Z_j|}\le\sum_j(e^{\lambda Z_j}+e^{-\lambda Z_j})$ give
\begin{align*}
 \exp\!\left\{\lambda\int\max_j|Z_j|\,d\nu\right\}
 &\le \int e^{\lambda\max_j|Z_j|}\,d\nu\\
 &\le\sum_{j=1}^M\bigl(\int e^{\lambda Z_j}\,d\nu
                       +\int e^{-\lambda Z_j}\,d\nu\bigr)\\
 &\le2M e^{\lambda^2\sigma^2/2}.
\end{align*}
Taking logarithms and optimizing in $\lambda$ proves the claim; the case
$\sigma=0$ is immediate.
\end{proof}

A centered Gaussian linear functional with second moment at most $\sigma^2$ satisfies the hypothesis.  The same is true for a Rademacher sum $Z=\sum_r\varepsilon_ra_r$ with
$\sum_ra_r^2\le\sigma^2$, by $\cosh u\le e^{u^2/2}$.

\begin{lemma}\label{lem:gaussian-tensor}
Let $\ell,d$ be positive integers and set
\[
 H_{\ell,d}:=\operatorname{Sym}^{\ell}(\mathbb R^d)
 \subset (\mathbb R^d)^{\otimes\ell}
\]
be endowed with the Hilbertian structure inherited from the full tensor
product.  Let $G$ be a standard Gaussian vector in $H_{\ell,d}$, that is,
if $(e_\nu)_\nu$ is any orthonormal basis of $H_{\ell,d}$ and
$(g_\nu)_\nu$ are independent standard real Gaussian variables, then
\[
 G=\sum_\nu g_\nu e_\nu.
\]
For every $\ell\ge1$,
\begin{equation}\label{eq:symmetric-gaussian-tensor}
 \int_{H_{\ell,d}} \|G\|_\varepsilon\,d\gamma_{H_{\ell,d}}(G)
 \le
 K\sqrt{d\log(e\ell)},
\end{equation}
where $K>0$ is an absolute constant; for instance, one may take $K=6$.
\end{lemma}

\begin{proof}
For symmetric tensors, Banach\'s diagonal principle reduces the injective norm to diagonal values.  For a continuous
$\ell$-linear form $T$, write
\[
 \|T\|_{\mathrm{op}}
 :=
 \sup\left\{
 |T(x_1,\ldots,x_\ell)|:
 \|x_r\|_2\le1,\ 1\le r\le\ell
 \right\}.
\]
If $T$ is symmetric, Banach's diagonal principle gives
\begin{equation}\label{eq:banach-diagonal}
 \|T\|_{\mathrm{op}}
 =
 \sup_{\|u\|_2\le1}|T(u,\ldots,u)|.
\end{equation}
Under the usual identification of a symmetric tensor with its associated
multilinear form, $\|T\|_{\mathrm{op}}=\|T\|_\varepsilon$.  The diagonal
identity~\eqref{eq:banach-diagonal} is Banach's theorem; see
\cite[Satz~I, \S~5]{Banach} and, in modern tensor notation,
\cite[Lemma~1(6), Eq.~(7)]{FriedlandWang}.

Fix a $1/(2\ell)$-net $\mathcal N$ of the Euclidean unit sphere
$S^{d-1}$.  We may choose it so that
\begin{equation}\label{eq:net-cardinality}
 |\mathcal N|
 \le
 (1+4\ell)^d.
\end{equation}
Indeed, put $\varepsilon=1/(2\ell)$ and take $\mathcal N$ to be a
maximal $\varepsilon$-separated subset of $S^{d-1}$.  Maximality makes
$\mathcal N$ an $\varepsilon$-net.  Moreover, the Euclidean balls of radius
$\varepsilon/2$ centered at the points of $\mathcal N$ are pairwise disjoint
and contained in the ball of radius $1+\varepsilon/2$.  Comparing volumes
gives
\[
 |\mathcal N|\left(\frac{\varepsilon}{2}\right)^d
 \le \left(1+\frac{\varepsilon}{2}\right)^d,
\]
and hence $|\mathcal N|\le(1+2/\varepsilon)^d=(1+4\ell)^d$.
Let $u\in S^{d-1}$ and choose $v\in\mathcal N$ with
\[
 \|u-v\|_2\le\frac1{2\ell}.
\]
By multilinearity,
\begin{align*}
 T(u,\ldots,u)-T(v,\ldots,v)
 &=
 \sum_{r=1}^{\ell}
 T(v,\ldots,v,
 \underbrace{u-v}_{r\text{-th place}},
 u,\ldots,u).
\end{align*}
All the vectors other than $u-v$ have Euclidean norm one.  By the
definition of the full operator norm,
\[
 |T(v,\ldots,v,u-v,u,\ldots,u)|
 \le
 \|T\|_{\mathrm{op}}\|u-v\|_2.
\]
Summing the $\ell$ terms and then using
\eqref{eq:banach-diagonal},
\[
 |T(u,\ldots,u)-T(v,\ldots,v)|
 \le
 \ell\|u-v\|_2\|T\|_{\mathrm{op}}
 \le
 \frac12\|T\|_{\mathrm{op}}.
\]
Taking the supremum over $u$ and using Banach's diagonal identity once
more gives
\begin{equation}\label{eq:net-reduction}
 \|T\|_\varepsilon
 =
 \|T\|_{\mathrm{op}}
 \le
 2\max_{v\in\mathcal N}|T(v,\ldots,v)|.
\end{equation}

Apply this to the Gaussian vector $G\in H_{\ell,d}$.  For $v\in S^{d-1}$,
the symmetric rank-one tensor $v^{\otimes\ell}$ belongs to $H_{\ell,d}$ and
\[
 \|v^{\otimes\ell}\|_2
 =
 \|v\|_2^\ell
 =
 1.
\]
Since $\gamma_{H_{\ell,d}}$ is the standard Gaussian measure on $H_{\ell,d}$,
\[
 \langle G,v^{\otimes\ell}\rangle
\]
is a centered Gaussian linear functional and
\[
 \int_{H_{\ell,d}} |\langle G,v^{\otimes\ell}\rangle|^2\,d\gamma_{H_{\ell,d}}(G)
 =
 \|v^{\otimes\ell}\|_2^2
 =
 1.
\]
No independence among the linear functionals indexed by $v$ is needed.
Lemma~\ref{lem:subgaussian-max} applies to their individual Gaussian moment bounds.  From~\eqref{eq:net-reduction},
\[
 \int_{H_{\ell,d}} \|G\|_\varepsilon\,d\gamma_{H_{\ell,d}}(G)
 \le
 2\int_{H_{\ell,d}}
 \max_{v\in\mathcal N}|\langle G,v^{\otimes\ell}\rangle|\,d\gamma_{H_{\ell,d}}(G)
 \le
 2\sqrt{2\log(2|\mathcal N|)}.
\]
Using~\eqref{eq:net-cardinality},
\[
 \log(2|\mathcal N|)
 \le
 \log2+d\log(1+4\ell)
 \le
 \frac92 d\log(e\ell).
\]
Thus $2\sqrt{2\log(2|\mathcal N|)}\le6\sqrt{d\log(e\ell)}$,
which gives~\eqref{eq:symmetric-gaussian-tensor} with $K=6$.
\end{proof}

\begin{lemma}\label{lem:increment-comparison}
Let $x,y\in\mathfrak X_{m,n}$ and let
$\eta,\theta\in\{-1,1\}$.  Then
\[
 \|\eta c(x)-\theta c(y)\|_2
 \le
 \sqrt{m!}\,
 \|\eta\Psi(x)-\theta\Psi(y)\|_2.
\]
\end{lemma}

\begin{proof}
Put
\[
 u_x:=x^{(1)}\otimes\cdots\otimes x^{(m)},
 \qquad
 u_y:=y^{(1)}\otimes\cdots\otimes y^{(m)},
\]
and consider the tensor
\[
 v:=\eta u_x-\theta u_y.
\]
The map which sends a tensor to its vector of orbit sums is linear.  Hence
\[
 c(v)=\eta c(x)-\theta c(y).
\]
Likewise, $P_{\sym}$ is linear, and therefore
\[
 P_{\sym}v
 =
 \eta P_{\sym}u_x-\theta P_{\sym}u_y
 =
 \eta\Psi(x)-\theta\Psi(y).
\]
Applying~\eqref{eq:increment-projection} directly to this tensor $v$ gives
\[
 \|\eta c(x)-\theta c(y)\|_2^2
 \le
 m!\,\|\eta\Psi(x)-\theta\Psi(y)\|_2^2.
\]
Taking square roots proves the lemma.
\end{proof}

\begin{lemma}\label{lem:comparison-principles}
Let $\Gamma$ be a nonempty finite set, let $H_1,H_2$ be finite-dimensional real Hilbert spaces, and let $\lambda\ge0$.

\begin{enumerate}[label=\textup{(\roman*)}]
\item Suppose that maps $a:\Gamma\to H_1$ and $b:\Gamma\to H_2$ satisfy
\[
 \|a(s)-a(t)\|_{H_1}
 \le \lambda\|b(s)-b(t)\|_{H_2}
 \qquad(s,t\in\Gamma).
\]
Then the Gaussian comparison theorem
\cite[Theorem~1, Eq.~(3)]{Vitale} gives
\[
 \int_{H_1} \sup_{s\in\Gamma}\langle g_1,a(s)\rangle\,d\gamma_{H_1}(g_1)
 \le
 \lambda\int_{H_2} \sup_{s\in\Gamma}\langle g_2,b(s)\rangle\,d\gamma_{H_2}(g_2).
\]

\item If $V\subset\mathbb R^N$ is finite and nonempty, then
\begin{equation}\label{eq:rad-gauss-comparison}
 \int_{\{-1,1\}^N} \sup_{v\in V}|\langle\varepsilon,v\rangle|\,d\mu_{\{1,\ldots,N\}}(\varepsilon)
 \le
 \sqrt{\frac\pi2}\,
 \int_{\mathbb R^N} \sup_{v\in V}|\langle g,v\rangle|\,d\gamma_{\mathbb R^N}(g).
\end{equation}
\end{enumerate}
\end{lemma}

\begin{proof}
Part~(i) is the usual Sudakov--Fernique theorem, stated here with the exact
metric hypothesis used later.  For part~(ii), write a standard Gaussian
vector coordinatewise as
\[
 g=(\varepsilon_1r_1,\ldots,\varepsilon_Nr_N),
\]
under the product decomposition of $\gamma_{\mathbb R^N}$ into signs and
absolute values.  The sign vector is distributed according to
$\mu_{\{1,\ldots,N\}}$, the absolute values are independent of the signs,
and
$\int_{\mathbb R}|t|\,d\gamma_{\mathbb R}(t)=\sqrt{2/\pi}$.  For fixed signs the map
\[
 (r_1,\ldots,r_N)\longmapsto
 \sup_{v\in V}\left|\sum_{j=1}^N\varepsilon_jr_jv_j\right|
\]
is convex.  Jensen's inequality in the variables $r_j$ gives
\[
 \int_{\mathbb R^N} \sup_{v\in V}|\langle g,v\rangle|\,d\gamma_{\mathbb R^N}(g)
 \ge
 \sqrt{\frac2\pi}\,
 \int_{\{-1,1\}^N} \sup_{v\in V}|\langle\varepsilon,v\rangle|\,d\mu_{\{1,\ldots,N\}}(\varepsilon),
\]
which is equivalent to~\eqref{eq:rad-gauss-comparison}.
\end{proof}

\begin{lemma}\label{lem:width-union}
Let $S_1,\ldots,S_M$ be nonempty bounded subsets of a finite-dimensional real
Hilbert space and assume
$\|s\|_2\le R_0$ for every $s\in\bigcup_jS_j$.  Then
\begin{equation}\label{eq:width-union}
 w\!\left(\bigcup_{j=1}^M S_j\right)
 \le
 \max_{1\le j\le M}w(S_j)
 +R_0\sqrt{2\log M}
\end{equation}
for $M\ge2$.  For $M=1$ the second term is omitted.
Empty sets may simply be omitted from the union.
\end{lemma}

\begin{proof}
The case $M=1$ is immediate, so assume $M\ge2$.  If $R_0=0$, every $S_j$
is contained in $\{0\}$ and the assertion is immediate.  Assume $R_0>0$.
Let
\[
 H_0:=\operatorname{span}\left(\bigcup_{j=1}^M S_j\right),
\]
let $g$ be standard Gaussian in $H_0$, and set
\[
 F_j(g):=\sup_{s\in S_j}|\langle g,s\rangle|.
\]
The projection of $g$ onto $\operatorname{span}S_j$ is standard Gaussian
there, and hence $\int_{H_0} F_j(g)\,d\gamma_{H_0}(g)=w(S_j)$.
For all $g,h$,
\[
 |F_j(g)-F_j(h)|
 \le R_0\|g-h\|_2,
\]
so Gaussian concentration
\cite[Proposition~2.1, Eq.~(2.7)]{LedouxGaussian} gives
\[
 \int_{H_0}\exp\!\left\{\lambda\left(F_j(g)-\int_{H_0} F_j\,d\gamma_{H_0}\right)\right\}\,d\gamma_{H_0}(g)
 \le \exp(\lambda^2R_0^2/2)
 \qquad(\lambda\ge0).
\]
Therefore
\begin{align*}
 \int_{H_0}\max_{1\le j\le M}F_j(g)\,d\gamma_{H_0}(g)
 &\le
 \max_j\int_{H_0} F_j(g)\,d\gamma_{H_0}(g)
 +\frac{\log M}{\lambda}
 +\frac{\lambda R_0^2}{2}.
\end{align*}
Taking $\lambda=\sqrt{2\log M}/R_0$ proves the claim.
\end{proof}

\begin{lemma}\label{lem:far-tail-absorb}
There is an absolute constant $K>0$ such that, whenever
\[
 m\ge3,
 \qquad
 1\le L_0\le\frac{m-3}{2},
\]
one has
\[
 \sqrt m\,e^{-m/4}
 \le
 K e^{-L_0/2}
 \sqrt{
 1+(L_0+1)\log\frac{em}{L_0+1}}.
\]
\end{lemma}

\begin{proof}
We split into two cases.

If $L_0\le m/4$, then
\[
 e^{-L_0/2}\ge e^{-m/8}.
\]
Also the square-root factor on the right is at least $1$.  Hence
\[
 \frac{\sqrt m\,e^{-m/4}}
 {e^{-L_0/2}
 \sqrt{1+(L_0+1)\log\frac{em}{L_0+1}}}
 \le
 \sqrt m\,e^{-m/8}.
\]
The function $s\mapsto\sqrt s\,e^{-s/8}$ is bounded on $[3,\infty)$.

Assume now that
\[
 \frac m4<L_0\le\frac{m-3}{2}.
\]
Then $L_0+1>m/4$, and
\[
 \frac{em}{L_0+1}\ge 2e.
\]
Therefore
\[
 1+(L_0+1)\log\frac{em}{L_0+1}
 \ge
 (L_0+1)\log(2e)
 \ge
 \frac m4\log(2e).
\]
Thus the square-root factor is at least $c\sqrt m$ for an absolute
$c>0$.  Since $L_0\le(m-3)/2$,
\[
 e^{-L_0/2}\ge e^{3/4}e^{-m/4}.
\]
Consequently
\[
 e^{-L_0/2}
 \sqrt{1+(L_0+1)\log\frac{em}{L_0+1}}
 \ge c e^{3/4}\sqrt m\,e^{-m/4},
\]
which proves the result after enlarging the absolute constant $K$.
\end{proof}

\begin{lemma}\label{lem:permanent-tail}
There is an absolute constant $K>0$ with the following property.  If
\[
 1\le L_0\le \frac{m-3}{2}
\]
and $n\ge m/e$, then
\[
 \int \sup_{L(x)\ge L_0}|A_{m,n}^{\varepsilon}(x)|\,d\mu(\varepsilon)
 \le
 K e^{-L_0/2}
 \sqrt{
 1+(L_0+1)\log\frac{em}{L_0+1}}
 \sqrt{m!}\,n^{(m+1)/2}.
\]
\end{lemma}

\begin{proof}
Consider a nonempty stratum $\mathcal X_k$ defined by
\eqref{eq:permanent-stratum}; empty strata are omitted from the sums below.
If $x\in\mathcal X_k$, then
Lemma~\ref{lem:permanent-control} gives
\[
 \sum_{\mathcal O}|c_{\mathcal O}(x)|^2
 \le
 m!n^m e^{-k}.
\]
Hence, for fixed $x\in\mathcal X_k$, the function on the sign cube
\[
 \varepsilon\longmapsto A_{m,n}^{\varepsilon}(x)
 =
 \sum_{\mathcal O}\varepsilon_{\mathcal O}c_{\mathcal O}(x)
\]
satisfies the exponential-moment hypothesis of Lemma~\ref{lem:subgaussian-max}
with $\sigma^2\le m!n^m e^{-k}$.  Therefore
\begin{equation}\label{eq:tail-stratum}
 \int \sup_{x\in\mathcal X_k}|A_{m,n}^{\varepsilon}(x)|\,d\mu(\varepsilon)
 \le
 K e^{-k/2}\sqrt{\log(2|\mathcal X_k|)}
 \sqrt{m!}\,n^{m/2}.
\end{equation}

Assume first that
\[
 k+1\le \frac{m-1}{2}.
\]
Lemma~\ref{lem:count} and~\eqref{eq:entropy-simple} give
\begin{align}
 \log(2|\mathcal X_k|)
 &\le \log(2m)+(n+m-1)\log2
      +n(k+1)\log\frac{e(m-1)}{k+1}.
 \label{eq:tail-entropy-raw}
\end{align}
Each term in \eqref{eq:tail-entropy-raw} is bounded by a multiple of $n$.  Since $n\ge m/e$,
we have $m\le en$, and hence
\[
 n+m-1\le (1+e)n.
\]
Also $m\ge3$ in the present argument, so $\log(2m)\le m\le en$.
Finally,
\[
 \log\frac{e(m-1)}{k+1}
 \le \log\frac{em}{k+1}.
\]
Consequently
\begin{align}
 \log(2|\mathcal X_k|)
 &\le en+(1+e)n\log2
      +n(k+1)\log\frac{em}{k+1}\notag\\
 &\le K n\left(
 1+(k+1)\log\frac{em}{k+1}
 \right).
 \label{eq:tail-entropy-reduced}
\end{align}
The constant in this entropy estimate is absolute.
Substituting this in~\eqref{eq:tail-stratum}, we obtain
\begin{equation}\label{eq:tail-small-k}
 \int \sup_{x\in\mathcal X_k}|A_{m,n}^{\varepsilon}(x)|\,d\mu(\varepsilon)
 \le
 K e^{-k/2}
 \sqrt{
 1+(k+1)\log\frac{em}{k+1}}
 \sqrt{m!}\,n^{(m+1)/2}.
\end{equation}

To sum these terms, put
\[
 f(s):=1+s\log\frac{em}{s},
 \qquad 1\le s\le m.
\]
On this interval,
\[
 f'(s)=\log\frac ms\ge0,
 \qquad
 \left(\frac{f(s)}s\right)'
 =-\frac1{s^2}-\frac1s<0.
\]
Thus $f$ is increasing, while $f(s)/s$ is decreasing.  Let
$k_0:=\lfloor L_0\rfloor$.  If $k=k_0+j$ remains in the present range, then
\[
 f(k+1)
 \le\frac{k+1}{k_0+1}f(k_0+1)
 \le(j+1)f(k_0+1).
\]
Consequently,
\begin{align*}
 \sum_{\substack{k\ge k_0\\ k+1\le(m-1)/2}}
 e^{-k/2}\sqrt{f(k+1)}
 &\le
 e^{-k_0/2}\sqrt{f(k_0+1)}
 \sum_{j=0}^{\infty}e^{-j/2}\sqrt{j+1}\\
 &\le K e^{-L_0/2}\sqrt{f(L_0+1)}.
\end{align*}
Here $k_0\ge L_0-1$, so
$e^{-k_0/2}\le e^{1/2}e^{-L_0/2}$; the absolute factor $e^{1/2}$ is absorbed
into $K$.  We also use $k_0+1\le L_0+1$ and the monotonicity of $f$.  Summing~\eqref{eq:tail-small-k} over the present range therefore gives
\[
 K e^{-L_0/2}
 \sqrt{
 1+(L_0+1)\log\frac{em}{L_0+1}}
 \sqrt{m!}\,n^{(m+1)/2}.
\]

Consider
\[
 k+1>\frac{m-1}{2}.
\]
Here we only use
\[
 |\mathcal X_k|\le2^{mn}.
\]
Equation~\eqref{eq:tail-stratum} gives
\[
 \int \sup_{x\in\mathcal X_k}|A_{m,n}^{\varepsilon}(x)|\,d\mu(\varepsilon)
 \le
 K e^{-k/2}\sqrt m\,
 \sqrt{m!}\,n^{(m+1)/2}.
\]
Hence
\[
 \sum_{k+1>(m-1)/2}
 \int \sup_{x\in\mathcal X_k}|A_{m,n}^{\varepsilon}(x)|\,d\mu(\varepsilon)
 \le
 K\sqrt m\,e^{-m/4}
 \sqrt{m!}\,n^{(m+1)/2}.
\]
Lemma~\ref{lem:far-tail-absorb} shows that this entire contribution is
bounded by
\[
 K e^{-L_0/2}
 \sqrt{
 1+(L_0+1)\log\frac{em}{L_0+1}}
 \sqrt{m!}\,n^{(m+1)/2}.
\]

Finally, for every choice of the signs,
\[
 \sup_{L(x)\ge L_0}|A_{m,n}^{\varepsilon}(x)|
 \le
 \sum_{k\ge k_0}
 \sup_{x\in\mathcal X_k}|A_{m,n}^{\varepsilon}(x)|.
\]
Integrating with respect to the Rademacher signs, the contribution of
$k+1\le(m-1)/2$ is bounded by the estimate obtained from
\eqref{eq:tail-small-k}, while the complementary range is bounded by
Lemma~\ref{lem:far-tail-absorb}.  Thus both ranges are bounded by a constant
multiple of
\[
 e^{-L_0/2}
 \sqrt{1+(L_0+1)\log\frac{em}{L_0+1}}
 \sqrt{m!}\,n^{(m+1)/2},
\]
which proves the lemma.
\end{proof}

\section{Regular and exceptional variables}\label{sec:prepA-barycentric}
\label{sec:regular-exceptional}

Fix a Hamming anchor $e_0$ and a scale $S\ge1+L$.  We call a vector regular
when it lies within $\eta/S$ of $e_0$, and exceptional otherwise.  The
energy bound leaves at most $O(S^3)$ exceptional indices.  We recenter the
regular vectors at the direction of their sum; their transverse components
then cancel exactly.  The exceptional variables remain in the same symmetric
expansion, where Vandermonde's identity absorbs their binomial mass.

Throughout this section the cube vectors are normalized to have Euclidean norm
one:
\[
 y_r:=n^{-1/2}x^{(r)}\in S^{n-1}.
\]
\begingroup
Let $r_0$ and $\delta_j$ be supplied by Lemma~\ref{lem:hamming}, and define
\[
 D:=\operatorname{diag}\bigl(x^{(r_0)}_1,\ldots,x^{(r_0)}_n\bigr),
 \qquad
 \widetilde x^{(j)}:=\delta_jDx^{(j)},
 \qquad
 \widetilde y_j:=n^{-1/2}\widetilde x^{(j)}.
\]
Since $\delta_{r_0}=1$, this sends the chosen anchor to the all-ones vector:
$\widetilde x^{(r_0)}=\mathbf1$.  Moreover,
\[
 \distH(\widetilde x^{(j)},\mathbf 1)
 =\distH(x^{(j)},\delta_jx^{(r_0)}).
\]
The normalized Gram entries transform as
\[
 \langle\widetilde y_r,\widetilde y_s\rangle
 =\delta_r\delta_s\langle y_r,y_s\rangle.
\]
Hence every term of the permanent is multiplied by
\[
 \prod_{r=1}^m\delta_r\delta_{\sigma(r)}
 =\left(\prod_{r=1}^m\delta_r\right)^2=1,
\]
so the Gram permanent, and therefore the permanent loss, is unchanged.  Henceforth the tildes are omitted: $x^{(j)}$ denotes the transformed cube vector and $y_j=n^{-1/2}x^{(j)}$ its Euclidean normalization.
With this convention, write
\[
 e_0=n^{-1/2}\mathbf 1.
\]
\endgroup
If the permanent loss is at most $L$, Lemma~\ref{lem:hamming} gives
\[
 \sum_{r=1}^m \distH(x^{(r)},\mathbf 1)\le nL.
\]
Since
\[
 \|y_r-e_0\|_2^2
 =\frac4n\distH(x^{(r)},\mathbf 1),
\]
we have the explicit energy estimate
\begin{equation}\label{eq:refinement-global-energy}
 \sum_{r=1}^m\|y_r-e_0\|_2^2\le4L.
\end{equation}
Fix a scale parameter
\begin{equation}\label{eq:refinement-scale}
 S\ge1+L.
\end{equation}
Taking $S$ larger than the loss allows us to use the same decomposition
throughout a permanent stratum.

Fix once and for all
\[
 0<\eta\le\min\{10^{-3},\eta_0\}.
\]
An index $r$ is called \emph{regular} if
\begin{equation}\label{eq:regular-threshold}
 \|y_r-e_0\|_2\le\frac{\eta}{S},
\end{equation}
and \emph{exceptional} otherwise.  Let $\mathcal R$ and $\mathcal E$ denote
the sets of regular and exceptional indices, respectively, and put
$d:=|\mathcal R|$ and $q:=|\mathcal E|$, so that $d+q=m$.

\begin{lemma}\label{lem:exceptional-degree}
With
\[
 Q_0:=4\eta^{-2},
\]
one has
\begin{equation}\label{eq:exceptional-degree}
 q\le Q_0S^3.
\end{equation}
\end{lemma}

\begin{proof}
Every exceptional index contributes more than $\eta^2/S^2$ to the left-hand
side of~\eqref{eq:refinement-global-energy}.  Therefore
\[
 q\frac{\eta^2}{S^2}\le4L\le4S,
\]
which is exactly~\eqref{eq:exceptional-degree}.
\end{proof}

Fix a numerical constant $M_0>2Q_0$.  If
\begin{equation}\label{eq:large-degree-refinement}
 m\ge M_0S^3,
\end{equation}
then Lemma~\ref{lem:exceptional-degree} gives
$q\le Q_0S^3<m/2$, and hence $d=m-q\ge m/2$.  For the regular indices set
\[
 Y_{\mathcal R}:=\sum_{r\in\mathcal R}y_r.
\]
Theorem~\ref{thm:abstract-moving-anchor}, applied to the regular family,
shows that $Y_{\mathcal R}\ne0$.  We may therefore set
\[
 e_{\mathcal R}:=\frac{Y_{\mathcal R}}{\|Y_{\mathcal R}\|_2}.
\]

\begin{lemma}
\label{lem:regular-sum-anchor}
Assume~\eqref{eq:large-degree-refinement}.  Then
\begin{align}
 \sum_{r\in\mathcal R}\|y_r-e_{\mathcal R}\|_2^2&\le4L,
 \label{eq:regular-anchor-energy}\\
 \|e_{\mathcal R}-e_0\|_2&\le\frac{\eta}{S},
 \label{eq:regular-anchor-distance}\\
 \max_{r\in\mathcal R}\|y_r-e_{\mathcal R}\|_2&\le\frac{2\eta}{S}.
 \label{eq:regular-individual-distance}
\end{align}
Writing
\[
 y_r=a_re_{\mathcal R}+z_r=a_r(e_{\mathcal R}+v_r),
 \qquad r\in\mathcal R,
\]
one has
\begin{equation}\label{eq:regular-longitudinal-positive}
 a_r\ge\frac12\qquad(r\in\mathcal R),
\end{equation}
as well as
\begin{equation}\label{eq:sum-anchor-cancellation}
 \sum_{r\in\mathcal R}z_r=0,
\end{equation}
and
\begin{align}
 \delta:=\max_{r\in\mathcal R}\|v_r\|_2&\le\frac{4\eta}{S},
 \label{eq:delta-regular}\\
 V:=\sum_{r\in\mathcal R}\|v_r\|_2^2&\le16L\le16S,
 \label{eq:V-regular}\\
 \left\|\sum_{r\in\mathcal R}v_r\right\|_2
 &\le\frac{16\eta L}{S}\le16\eta.
 \label{eq:P1-bounded}
\end{align}
\end{lemma}

\begin{proof}
The regular family satisfies the cap and energy hypotheses of
Theorem~\ref{thm:abstract-moving-anchor}, and $0\le L\le S$.  Moreover,
$d\ge m/2\ge1$.  Applying that theorem to $(y_r)_{r\in\mathcal R}$, with
$Y=Y_{\mathcal R}$ and $e=e_{\mathcal R}$, gives all the displayed
conclusions.
\end{proof}

\subsection{Exceptional levels}

From now on write $e:=e_{\mathcal R}$ and
\[
 C_{\mathcal R}:=\prod_{r\in\mathcal R}a_r
 =\prod_{r\in\mathcal R}(1+\|v_r\|_2^2)^{-1/2}.
\]
\begingroup
For $s\in\mathcal E$, define the longitudinal and transverse components by
\[
 b_s:=\langle y_s,e\rangle,
 \qquad
 w_s:=y_s-b_se.
\]
Then $w_s\perp e$ and $y_s=b_se+w_s$.
\endgroup
For $0\le j\le q$, put
\[
 B_j:=
 \sum_{\substack{J\subset\mathcal E\\ |J|=j}}
 \left(\prod_{s\in\mathcal E\setminus J}b_s\right)
 \widehat\bigotimes_{s\in J}w_s.
\]
The next estimate replaces the naive factor $2^q$ by its binomial
square-root at each level.

\begin{lemma}\label{lem:Bj-sharp}
For every $0\le j\le q$,
\begin{equation}\label{eq:Bj-sharp}
 \|B_j\|_\pi\le \sqrt{\binom qj}.
\end{equation}
\end{lemma}

\begin{proof}
Put $c_s:=\|w_s\|_2$.  Since $y_s$ is a unit vector and $w_s\perp e$,
$c_s^2+b_s^2=1$.  By the triangle inequality for the projective norm,
\[
 \|B_j\|_\pi
 \le
 \sum_{|J|=j}
 \prod_{s\notin J}|b_s|\prod_{s\in J}c_s.
\]
Cauchy--Schwarz gives
\begin{align*}
 \|B_j\|_\pi^2
 &\le \binom qj
 \sum_{|J|=j}
 \prod_{s\notin J}b_s^2\prod_{s\in J}c_s^2\\
 &\le \binom qj
 \prod_{s\in\mathcal E}(c_s^2+b_s^2)
 =\binom qj.
\end{align*}
\end{proof}

\begin{lemma}
\label{lem:slot-invariance}
Fix $q$ and write $d=m-q$.  For each exceptional set $\mathcal E$ of
cardinality $q$, let $\mathcal C(\mathcal E)$ be a class of $m$-tuples.
Assume that, for every $\sigma\in S_m$,
\[
 (y_1,\ldots,y_m)\in\mathcal C(\mathcal E)
 \quad\Longleftrightarrow\quad
 (y_{\sigma^{-1}(1)},\ldots,y_{\sigma^{-1}(m)})
 \in\mathcal C(\sigma\mathcal E).
\]
Define
\[
 \mathscr P(\mathcal E):=
 \left\{P_{\sym}^{(m)}(y_1\otimes\cdots\otimes y_m):
 (y_1,\ldots,y_m)\in\mathcal C(\mathcal E)\right\}.
\]
Then $\mathscr P(\mathcal E)$ is independent of $\mathcal E$.
\end{lemma}

\begin{proof}
If a permutation $\sigma$ carries one exceptional set to another, relabelling
by $\sigma$ is a bijection of the corresponding tuple classes, while
$P_{\sym}^{(m)}U_\sigma=P_{\sym}^{(m)}$.  Their images are therefore equal.
\end{proof}

Lemma~\ref{lem:slot-invariance} applies whenever a regular--exceptional class is defined by the designation of the factors and by quantities invariant under a simultaneous relabeling of the $m$ slots.

For a unit vector $e$ and $0\le h\le m$, set
\[
 J_{h,e}^{(m)}:\operatorname{Sym}^h(e^\perp)
 \longrightarrow\operatorname{Sym}^m(\R^n),
 \qquad
 J_{h,e}^{(m)}(w):=
 P_{\sym}^{(m)}\bigl(w\otimes e^{\otimes(m-h)}\bigr).
\]

\begin{lemma}
\label{lem:joint-regular-exceptional}
Let $m=d+q$.  Relative to the regular sum anchor $e=e_{\mathcal R}$, the normalized
symmetrized product has the exact expansion
\begin{equation}\label{eq:joint-regular-exceptional}
 P_{\sym}^{(m)}(y_1\otimes\cdots\otimes y_m)
 =
 \sum_{h=0}^m J_{h,e}^{(m)}(W_h),
\end{equation}
where
\begin{equation}\label{eq:Wh-definition}
 W_h
 :=C_{\mathcal R}\sum_{\ell+j=h}e_\ell(v)\widehat\otimes B_j.
\end{equation}
Moreover,
\begin{equation}\label{eq:total-level-normalization}
 \|J_{h,e}^{(m)}(W_h)\|_2^2
 =\frac{\|W_h\|_2^2}{\binom mh}.
\end{equation}
The ranges of $J_{h,e}^{(m)}$ are mutually orthogonal as $h$ varies.
\end{lemma}

\begin{proof}
For fixed $q$, the defining regular--exceptional conditions are invariant
under a simultaneous permutation of the $m$ slots.  Lemma~\ref{lem:slot-invariance}
therefore allows us to replace the actual exceptional set by the canonical
sets
$\mathcal R_0=\{1,\ldots,d\}$ and
$\mathcal E_0=\{d+1,\ldots,m\}$.  Hence no union over the
$\binom mq$ possible exceptional positions is needed.

Expand
\[
 y_r=a_r(e+v_r),\quad r\in\mathcal R_0,
 \qquad
 y_s=b_se+w_s,\quad s\in\mathcal E_0.
\]
For subsets $I\subset\mathcal R_0$ and $J\subset\mathcal E_0$, write
$v_I:=\widehat\bigotimes_{r\in I}v_r$ and
$w_J:=\widehat\bigotimes_{s\in J}w_s$, with the empty product equal to $1$.
Multilinearity gives the exact identity
\[
 P_{\sym}^{(m)}(y_1\otimes\cdots\otimes y_m)
 =C_{\mathcal R}
 \sum_{I\subset\mathcal R_0}
 \sum_{J\subset\mathcal E_0}
 \left(\prod_{s\in\mathcal E_0\setminus J}b_s\right)
 P_{\sym}^{(m)}
 \bigl(v_I\otimes w_J\otimes e^{\otimes(m-|I|-|J|)}\bigr).
\]
Each ordered choice of transverse factors occurs once before
symmetrization, so the expansion carries no additional combinatorial coefficient.

Fix $\ell=|I|$, $j=|J|$ and $h=\ell+j$.  Since
$P_{\sym}^{(m)}(P_{\sym}^{(h)}\otimes I) =P_{\sym}^{(m)}$, the sum over all
such $I,J$ is
\[
 J_{h,e}^{(m)}\!\left(
 C_{\mathcal R}
 e_\ell(v)\widehat\otimes B_j
 \right).
\]
Summing over $\ell+j=h$ proves
\eqref{eq:joint-regular-exceptional}--\eqref{eq:Wh-definition}.  Canonical
relabelling accounts for the exceptional positions, while the choices at
fixed $(\ell,j)$ are already contained in $e_\ell(v)$ and $B_j$.

\begingroup
Let $w\in\operatorname{Sym}^h(e^\perp)$.  For each
$A\subset\{1,\ldots,m\}$ with $|A|=h$, let $\iota_A(w)$ denote the tensor
obtained by placing the $h$ tensor factors of $w$ in the positions indexed by
$A$ and placing $e$ in the remaining positions; this is first defined on
simple tensors and then extended linearly and continuously.  Since $w$ is
symmetric,
\[
 J_{h,e}^{(m)}(w)
 =\binom mh^{-1}
  \sum_{\substack{A\subset\{1,\ldots,m\}\\|A|=h}}\iota_A(w).
\]
Each map $\iota_A$ is an isometry.  If $A\ne B$, choose a position in the
symmetric difference of $A$ and $B$; at that position one tensor has a factor
in $e^\perp$ and the other has the factor $e$.  Hence the ranges of
$\iota_A$ and $\iota_B$ are orthogonal.  Therefore
\[
 \bigl\|J_{h,e}^{(m)}(w)\bigr\|_2^2
 =\binom mh\binom mh^{-2}\|w\|_2^2
 =\frac{\|w\|_2^2}{\binom mh}.
\]
This proves \eqref{eq:total-level-normalization}.  The subspaces corresponding
to different values of $h$ are likewise orthogonal, because they contain
different numbers of factors from $e^\perp$.  This proves the last assertion.
\endgroup
\end{proof}

\begin{lemma}\label{lem:binomial-level-sum}
Let $t_S=1+c/S$, where $c>0$ is the constant in
Proposition~\ref{prop:abstract-moving-estimates}.  There are absolute
constants $C_*,C>0$ such that, whenever $S\ge1$ and $d\ge C_*S^2$,
\begin{equation}\label{eq:binomial-level-sum}
 \sum_{\ell=1}^{d}\frac{t_S^{-\ell}}{\sqrt{\binom d\ell}}
 \le \frac{C}{\sqrt d}.
\end{equation}
\end{lemma}

\begin{proof}
The terms $\ell=1$ and $\ell=d-1$ are bounded by $d^{-1/2}$.  For
$2\le\ell\le d-2$ we have
\[
 \binom d\ell\ge\binom d2.
\]
Since $t_S=1+c/S$,
\[
 \sum_{\ell\ge2}t_S^{-\ell}
 =\frac{t_S^{-2}}{1-t_S^{-1}}
 =\frac{1}{t_S(t_S-1)}
 =\frac{S}{c\,t_S}
 \le \frac{S}{c}.
\]
Therefore
\[
 \sum_{\ell=2}^{d-2}
 \frac{t_S^{-\ell}}{\sqrt{\binom d\ell}}
 \le \frac{C}{d}\sum_{\ell\ge2}t_S^{-\ell}
 \le C\frac{S}{d}.
\]
The hypothesis $d\ge C_*S^2$ gives
\[
 \frac{S}{d}\le\frac{1}{\sqrt{C_*}}\frac1{\sqrt d},
\]
so the middle range is bounded by $C/\sqrt d$.  It remains to treat
$\ell=d$.  We may assume $0<c<1$.  Since $0<c/S\le1$, the inequality
$\log(1+u)\ge u/2$ for $0\le u\le1$, applied with $u=c/S$, gives
\[
 t_S^{-d}\le \exp\!\left(-\frac{cd}{2S}\right).
\]
After increasing $C_*$, this is at most $d^{-1/2}$ for every $S\ge1$ and
$d\ge C_*S^2$.  Adding the bounds for $\ell=1,d-1$, the middle range, and $\ell=d$ proves
\eqref{eq:binomial-level-sum}.
\end{proof}

\begin{lemma}
\label{lem:fixed-anchor-joint-width}
There are absolute constants $c_*,C_*,C>0$ with the following property.
Let $S\ge1$ and suppose
\[
 q\le Q_0S^3,
 \qquad
 d=m-q\ge\frac m2,
 \qquad
 d\ge C_*S^2,
\]
and
\begin{equation}\label{eq:joint-admissibility}
 (q+1)\sqrt{\log(em)}\le c_*\sqrt d.
\end{equation}
Fix a unit vector $e\in\mathbb R^n$.  Consider all symmetrized products in
\eqref{eq:joint-regular-exceptional} for which
$v_1,\ldots,v_d\in e^\perp$ satisfy
\[
 \max_r\|v_r\|_2\le\frac{4\eta}{S},\qquad
 \sum_r\|v_r\|_2^2\le16S,\qquad
 \left\|\sum_rv_r\right\|_2\le16\eta,
\]
while the $q$ exceptional factors are arbitrary unit vectors, and set
\[
 C_{\mathcal R}(v):=
 \prod_{r=1}^d(1+\|v_r\|_2^2)^{-1/2}.
\]
This class has Gaussian width at most $C\sqrt n$.
\end{lemma}

\begin{proof}
If $n=1$, the transverse space $e^\perp$ is zero-dimensional, and the
class lies in the segment $[-1,1]e^{\otimes m}$, whose Gaussian width is
$\sqrt{2/\pi}$.  We may therefore assume $n\ge2$.
For $0\le\ell\le d$, $0\le j\le q$, and $h:=\ell+j\ge1$, let
$\mathfrak W_{\ell,j}$ be the class of elements
\[
 J_{h,e}^{(m)}\!\left(
 C_{\mathcal R}(v)e_\ell(v)\widehat\otimes B_j
 \right)
\]
arising from admissible data.  Every element under consideration belongs to
the Minkowski sum of these pairwise classes together with its scalar $h=0$
component.  For bounded sets $E,F$ in a Hilbert space,
\[
 w(E+F)
 =\int\sup_{u\in E,\,v\in F}|\langle g,u+v\rangle|\,d\gamma(g)
 \le w(E)+w(F).
\]
Thus it is enough to estimate $w(\mathfrak W_{\ell,j})$ and sum in
$\ell$ and $j$.

Let $G_h$ be standard Gaussian in the $h$th orthogonal range of
$J_{h,e}^{(m)}$.  By \eqref{eq:total-level-normalization}, for
$w\in\operatorname{Sym}^h(e^\perp)$,
\[
 \int\bigl|\langle G_h,J_{h,e}^{(m)}(w)\rangle\bigr|^2\,d\gamma(G_h)
 =\frac{\|w\|_2^2}{\binom mh}.
\]
Hence the covariance induced on $\operatorname{Sym}^h(e^\perp)$ by
$J_{h,e}^{(m)}$ is $\binom mh^{-1}$ times the Hilbert covariance.
Equivalently, after pullback the corresponding Gaussian linear functional is
$\binom mh^{-1/2}\langle G,w\rangle$, where $G$ is standard Gaussian in
$\operatorname{Sym}^h(e^\perp)$.  For each admissible pair $v,B_j$, duality and
\eqref{eq:symmetric-projective-submultiplicative} give
\begin{align*}
 \bigl|
 \langle G,
 C_{\mathcal R}e_\ell(v)\widehat\otimes B_j
 \rangle
 \bigr|
 &\le
 \|G\|_\varepsilon\,
 C_{\mathcal R}
 \|e_\ell(v)\widehat\otimes B_j\|_\pi\\
 &\le
 \|G\|_\varepsilon\,
 C_{\mathcal R}\|e_\ell(v)\|_\pi\,\|B_j\|_\pi.
\end{align*}
Taking the supremum, integrating with respect to Gaussian measure, and then applying
Lemma~\ref{lem:gaussian-tensor}, gives
\begin{equation}\label{eq:joint-pair-preliminary}
 w(\mathfrak W_{\ell,j})
 \le
 \frac{C\sqrt{n\log(eh)}}{\sqrt{\binom mh}}
 \sup C_{\mathcal R}\|e_\ell(v)\|_\pi\,\|B_j\|_\pi.
\end{equation}
The hypotheses on $v$ are precisely those of
Proposition~\ref{prop:abstract-moving-estimates}.  Hence
\[
 C_{\mathcal R}(v)
 \sum_{r\ge0}\|e_r(v)\|_\pi t_S^r\le C,
\]
and in particular
$C_{\mathcal R}(v)\|e_\ell(v)\|_\pi\le Ct_S^{-\ell}$.  Meanwhile,
Lemma~\ref{lem:Bj-sharp} gives
$\|B_j\|_\pi\le\sqrt{\binom qj}$.  Hence, with
$\omega_{\ell,j}:=w(\mathfrak W_{\ell,j})$,
\begin{equation}\label{eq:joint-pair-width}
 \omega_{\ell,j}\le
 C\sqrt{n\log(eh)}\,t_S^{-\ell}
 \sqrt{\frac{\binom qj}{\binom m{\ell+j}}}.
\end{equation}

For $\ell\ge1$, Vandermonde's identity gives
\[
 \binom m{\ell+j}
 =\sum_{r}\binom dr\binom q{\ell+j-r}
 \ge \binom d\ell\binom qj.
\]
Consequently the coefficient in \eqref{eq:joint-pair-width} satisfies the
pointwise bound
\begin{equation}\label{eq:vandermonde-pointwise-absorption}
 \sqrt{\frac{\binom qj}{\binom m{\ell+j}}}
 \le \frac1{\sqrt{\binom d\ell}}.
\end{equation}
Hence the factor $\binom qj$ is absorbed at each fixed level before
summing in $j$. Since
$\log(e(\ell+j))\le\log(em)$,
Lemma~\ref{lem:binomial-level-sum} yields
\begin{align*}
 \sum_{\substack{\ell\ge1\\0\le j\le q}}\omega_{\ell,j}
 &\le
 C\sqrt n\,(q+1)\sqrt{\log(em)}
 \sum_{\ell=1}^d
 \frac{t_S^{-\ell}}{\sqrt{\binom d\ell}}\\
 &\le
 C\sqrt n\,
 \frac{(q+1)\sqrt{\log(em)}}{\sqrt d}
 \le C\sqrt n
\end{align*}
by \eqref{eq:joint-admissibility}.

It remains to consider $\ell=0$.  Put
\[
 \rho:=\frac qd=\frac q{m-q}.
\]
Choose $c_*\le1/2$.  Since \eqref{eq:joint-admissibility} and
$\log(em)\ge1$ imply $q\le c_*\sqrt d$, one has
$\rho\le c_*/\sqrt d\le1/2$.  For $1\le j\le q$,
\[
 \frac{\binom qj}{\binom mj}
 =
 \prod_{r=0}^{j-1}\frac{q-r}{m-r}
 \le
 \left(\frac q{m-q}\right)^j
 =
 \rho^j.
\]
The purely exceptional levels are therefore bounded by
\[
 C\sqrt n
 \sum_{j\ge1}
 \sqrt{\log(ej)}\,\rho^{j/2}
 \le C\sqrt n.
\]
For $h=0$ one has $W_0=C_{\mathcal R}B_0$, so $|W_0|\le1$.
The elements $J_{0,e}^{(m)}(W_0)=W_0e^{\otimes m}$ lie in a one-dimensional
segment and have Gaussian width at most $\int_{\mathbb R}|t|\,d\gamma_{\mathbb R}(t)$.  Summing the levels
completes the proof.
\end{proof}

\section{Proof of the upper bound in Theorem~A}\label{sec:proofA}

Lemma~\ref{lem:fixed-anchor-joint-width} treats a fixed regular sum
anchor.  In the application we group configurations by their sum direction
and exceptional degree and use Lemma~\ref{lem:width-union}.  On the $k$th
permanent stratum the Hilbert radius is at most $e^{-k/2}$, which compensates
for the entropy of this finite union.  The required parameter ranges are
collected in Lemma~\ref{lem:parameter-compatibility}.

Fix once and for all constants $c_*,C_*>0$ for which
Lemma~\ref{lem:fixed-anchor-joint-width} holds.

\begin{lemma}
\label{lem:parameter-compatibility}
Let
\[
 T_m:=\log\log(m+e^e),
 \qquad L_0:=2T_m+8,
 \qquad S_k:=k+2.
\]
There is an absolute integer $m_0$ such that, whenever $m\ge m_0$,
\[
 1\le L_0\le\frac{m-3}{2},
\]
and, for every pair of integers
\[
 0\le k\le\lceil L_0\rceil,
 \qquad 0\le q\le Q_0S_k^3,
\]
with $d:=m-q$, one has
\begin{align*}
 m&\ge M_0S_k^3,
 &d&\ge\frac m2,
 &d&\ge C_*S_k^2,\\
 (q+1)\sqrt{\log(em)}&\le c_*\sqrt d.
\end{align*}
\end{lemma}

\begin{proof}
Uniformly in $k$,
$S_k\le2T_m+11=O(\log\log m)$, and hence $S_k^3=o(m)$.  Thus
$m\ge M_0S_k^3$ for large $m$.  Since $M_0>2Q_0$,
\[
 q\le Q_0S_k^3\le\frac{Q_0}{M_0}m<\frac m2,
\]
so $d\ge m/2$.  The inequality $d\ge C_*S_k^2$ follows from
$m/S_k^2\to\infty$, while
\[
 \frac{(q+1)\sqrt{\log(em)}}{\sqrt d}
 \le C\frac{(1+S_k^3)\sqrt{\log(em)}}{\sqrt m}
 \longrightarrow0
\]
uniformly in $k$.  Finally, $L_0=O(\log\log m)=o(m)$, which gives the
required tail range.
\end{proof}

\begin{lemma}\label{lem:moving-anchor-entropy}
Fix a loss stratum
\[
 e^{-(k+1)}<\alpha(x)\le e^{-k}
\]
and a regular--exceptional decomposition with $d$ regular variables.  After
the Hamming normalization of Lemma~\ref{lem:hamming}, suppose that
\[
 \sum_{r\in\mathcal R}\distH(x^{(r)},\mathbf 1)
 \le n(k+1).
\]
Set
\[
 Y_{\mathcal R}:=n^{-1/2}\sum_{r\in\mathcal R}x^{(r)},
 \qquad
 e_{\mathcal R}:=\frac{Y_{\mathcal R}}{\|Y_{\mathcal R}\|_2}
\]
whenever $Y_{\mathcal R}\ne0$.  Then the number of possible directions
$e_{\mathcal R}$ is at most
\begin{equation}\label{eq:moving-anchor-entropy-lemma}
 M_k
 \le
 \binom{n+\lfloor n(k+1)\rfloor}{n}
 \le \exp\{Cn\log(e+k)\}
\end{equation}
for an absolute constant $C$.
\end{lemma}

\begin{proof}
For $1\le i\le n$, let $\nu_i$ be the number of regular vectors whose $i$th
coordinate equals $-1$.  Then
\[
 (Y_{\mathcal R})_i=n^{-1/2}(d-2\nu_i),
\]
so the integer vector $\nu=(\nu_1,\ldots,\nu_n)$ determines
$Y_{\mathcal R}$ and hence its normalized direction.  Moreover,
\[
 \sum_{i=1}^n\nu_i
 =\sum_{r\in\mathcal R}\distH(x^{(r)},\mathbf1)
 \le n(k+1).
\]
The number of nonnegative integer vectors satisfying this inequality is
\[
 \binom{n+\lfloor n(k+1)\rfloor}{n},
\]
by stars and bars, and the last estimate in
\eqref{eq:moving-anchor-entropy-lemma} follows from
\[
 \binom{N}{n}\le \left(\frac{eN}{n}\right)^n,
 \qquad N\ge n,
\]
with $N=n+\lfloor n(k+1)\rfloor$.  More explicitly,
\[
 \frac{N}{n}
 =1+\frac{\lfloor n(k+1)\rfloor}{n}
 \le k+2,
\]
and hence
\[
 \binom{n+\lfloor n(k+1)\rfloor}{n}
 \le \bigl(e(k+2)\bigr)^n
 =\exp\{n[1+\log(k+2)]\}
 \le \exp\{Cn\log(e+k)\}.
\]
Different vectors $\nu$ may yield the same normalized direction, which can
only decrease the number of anchors.
\end{proof}

\begin{proposition}
\label{prop:regular-exceptional-low-loss}
There is an absolute constant $C>0$ such that, for all sufficiently large
$m$ and every $n\ge m/e$, with
\begin{equation}\label{eq:L0-definition}
 T_m:=\log\log(m+e^e),
 \qquad
 L_0:=2T_m+8,
\end{equation}
one has
\begin{equation}\label{eq:low-loss-final}
 \int \sup_{L(x)\le L_0}|A_{m,n}^{\varepsilon}(x)|\,d\mu(\varepsilon)
 \le
 C\sqrt{m!}\,n^{(m+1)/2}.
\end{equation}
\end{proposition}

\begin{proof}
Cover the low-loss region by the strata
\begin{equation}\label{eq:low-loss-strata}
 e^{-(k+1)}<\alpha(x)\le e^{-k},
 \qquad
 0\le k\le\lceil L_0\rceil,
\end{equation}
and put $S_k:=k+2$.  On the $k$th stratum, $L(x)<k+1$.
We perform the regular--exceptional decomposition with the common scale
\[
S=S_k=k+2,
\]
which is admissible by~\eqref{eq:refinement-scale}.  Uniformly for
$0\le k\le\lceil L_0\rceil$, Lemma~\ref{lem:parameter-compatibility}
gives the large-degree condition.  Lemma~\ref{lem:exceptional-degree} first gives
\[
 q\le Q_0S_k^3.
\]
Applying the compatibility lemma to this $q$ yields
\[
 d=m-q\ge\frac m2,
 \qquad
 d\ge C_*S_k^2,
 \qquad
 (q+1)\sqrt{\log(em)}\le c_*\sqrt d.
\]
Lemma~\ref{lem:regular-sum-anchor} and then
Lemma~\ref{lem:fixed-anchor-joint-width} therefore apply throughout the
stratum.  The latter bounds every class with fixed sum anchor and fixed
exceptional degree by $C\sqrt n$ in the normalized symmetric Hilbert space.  Notice that
the exceptional vectors themselves are allowed to vary freely in this
class.

\begingroup
For each
$0\le q\le\lfloor Q_0S_k^3\rfloor$, let $\mathcal A_{k,q}$ be the set of
regular sum directions that arise after the Hamming normalization on the
$k$th stratum.  For $e\in\mathcal A_{k,q}$, let
$\mathfrak S_{k,q,e}$ denote the set of normalized symmetrized products
arising from Hamming-normalized configurations in that stratum with exceptional
degree $q$ and regular sum anchor $e$.  By Lemma~\ref{lem:slot-invariance}, symmetrization makes the choice of
the exceptional positions immaterial. Lemma~\ref{lem:fixed-anchor-joint-width} gives
\[
 w(\mathfrak S_{k,q,e})\le C\sqrt n,
\]
and \eqref{eq:alpha} gives
\[
 \sup_{u\in\mathfrak S_{k,q,e}}\|u\|_2\le e^{-k/2}.
\]
By Lemma~\ref{lem:moving-anchor-entropy}, uniformly in $q$,
\begin{equation}\label{eq:moving-anchor-count}
 |\mathcal A_{k,q}|\le M_k\le\exp\{Cn\log(e+k)\}.
\end{equation}

To undo the Hamming normalization, let $\mathcal D_n$ be the set of the
$2^n$ diagonal sign matrices on $\mathbb R^n$.  If
$\widetilde y_j=\delta_jDy_j$, where $D$ is the diagonal sign matrix
determined by the cube anchor, then
\[
 P_{\sym}^{(m)}(y_1\otimes\cdots\otimes y_m)
 =\left(\prod_{j=1}^m\delta_j\right)
 D^{\otimes m}
 P_{\sym}^{(m)}(\widetilde y_1\otimes\cdots\otimes\widetilde y_m).
\]
Let $\mathfrak T_k$ be the set of normalized symmetrized products in the $k$th
stratum.  Since the distinguished anchor slot is immaterial after
symmetrization and the scalar $\prod_j\delta_j$ is absorbed by adjoining the
negative of each class, we have the explicit covering
\[
 \mathfrak T_k
 \subset
 \bigcup_{q=0}^{\lfloor Q_0S_k^3\rfloor}
 \ \bigcup_{e\in\mathcal A_{k,q}}
 \ \bigcup_{D\in\mathcal D_n}
 D^{\otimes m}
 \bigl(\mathfrak S_{k,q,e}\cup(-\mathfrak S_{k,q,e})\bigr).
\]
The maps $D^{\otimes m}$ are orthogonal, and adjoining the negative of a set
does not change its absolute Gaussian width or its Hilbert radius.  Therefore
every class in this covering has width at most $C\sqrt n$ and radius at most
$e^{-k/2}$.  The number $N_k$ of sets in the covering satisfies
\[
 N_k\le M_k\bigl(1+\lfloor Q_0S_k^3\rfloor\bigr)2^n.
\]
\endgroup
Using \eqref{eq:moving-anchor-count} and $S_k=k+2$, we obtain
\begin{align*}
 \log N_k
 &\le Cn\log(e+k)
   +\log\bigl(1+Q_0(k+2)^3\bigr)+n\log2\\
 &\le Cn\log(e+k),
\end{align*}
where the last constant is absolute because $n\ge1$ and
$\log(1+Q_0(k+2)^3)\le C\log(e+k)$.
\begingroup
\color{black}
Each member of the union has width at most $C\sqrt n$, Hilbert radius at most $e^{-k/2}$, and the union has $N_k$ members.  Hence Lemma~\ref{lem:width-union} gives
\endgroup
\begin{align}
 w(\mathfrak T_k)
 &\le C\sqrt n+e^{-k/2}\sqrt{2\log N_k}\notag\\
 &\le C\sqrt n
    +Ce^{-k/2}\sqrt{n\log(e+k)}.
 \label{eq:one-stratum-width-before-absorption}
\end{align}
The second term in \eqref{eq:one-stratum-width-before-absorption} is also $O(\sqrt n)$ uniformly in $k$.  Indeed,
put $u=e+k\ge e$.  Then
\[
 e^{-k}\log(e+k)=e^e e^{-u}\log u,
\]
and the function $u\mapsto e^{-u}\log u$ is bounded on $[e,\infty)$.
Hence
\[
 e^{-k/2}\sqrt{\log(e+k)}\le C,
\]
and \eqref{eq:one-stratum-width-before-absorption} becomes
\begin{equation}\label{eq:one-stratum-open-width}
 w(\mathfrak T_k)\le C\sqrt n.
\end{equation}
Every element in these normalized strata has Hilbert norm at most one.  There are
$J:=1+\lceil L_0\rceil$ strata.  Applying
Lemma~\ref{lem:width-union} once more, now using
\eqref{eq:one-stratum-open-width}, gives
\begin{align}
 w\!\left(\bigcup_{0\le k\le\lceil L_0\rceil}\mathfrak T_k\right)
 &\le \max_{0\le k\le\lceil L_0\rceil}w(\mathfrak T_k)
      +\sqrt{2\log J}\notag\\
 &\le C\sqrt n+\sqrt{2\log(1+\lceil L_0\rceil)}.
 \label{eq:all-low-strata-width-before-absorption}
\end{align}
Since $L_0=O(\log\log m)$,
\[
 \log(1+\lceil L_0\rceil)=O(\log\log\log m).
\]
On the other hand $n\ge m/e$, so $\sqrt n\ge\sqrt{m/e}$.  Thus, after
enlarging the absolute constant to cover the finitely many small values of
$m$,
\[
 \sqrt{2\log(1+\lceil L_0\rceil)}\le C\sqrt n.
\]
Substitution in \eqref{eq:all-low-strata-width-before-absorption} proves
\begin{equation}\label{eq:all-low-strata-open-width}
 w\!\left(\bigcup_{0\le k\le\lceil L_0\rceil}\mathfrak T_k\right)
 \le C\sqrt n.
\end{equation}

Since
\[
 \Psi(x)=n^{m/2}\Phi(x),
\]
Gaussian width is homogeneous, and therefore
\begin{align}
 w\bigl(\{\Psi(x):L(x)\le L_0\}\bigr)
 &=n^{m/2}
   w\bigl(\{\Phi(x):L(x)\le L_0\}\bigr)\notag\\
 &\le n^{m/2}\,C\sqrt n\notag\\
 &=C n^{(m+1)/2}.
 \label{eq:Psi-low-width}
\end{align}
To pass back to the orbit-sign family, introduce
the signed index set
\[
 \mathcal I:=\{(x,\eta):L(x)\le L_0,\ \eta\in\{-1,1\}\}.
\]
The signed enlargement gives, explicitly,
\[
 \sup_{(x,\eta)\in\mathcal I}\langle g,\eta c(x)\rangle
 =\sup_{L(x)\le L_0}\max_{\eta\in\{-1,1\}}\eta\langle g,c(x)\rangle
 =\sup_{L(x)\le L_0}|\langle g,c(x)\rangle|.
\]
The same identity holds with $c(x)$ replaced by $\Psi(x)$.  Lemma~\ref{lem:increment-comparison} gives, for any two indices $(x,\eta)$ and
$(y,\theta)$,
\[
 \|\eta c(x)-\theta c(y)\|_2
 \le \sqrt{m!}\,
 \|\eta\Psi(x)-\theta\Psi(y)\|_2.
\]
Set
\[
 H_c:=\mathbb R^{\mathcal O_{m,n}},
 \qquad
 H_\Psi:=\operatorname{Sym}^m(\mathbb R^n).
\]
For $\xi=(x,\eta)$ and $\zeta=(y,\theta)$ in $\mathcal I$, define
\[
 X_\xi(g):=\langle g,\eta c(x)\rangle\quad(g\in H_c),
 \qquad
 Y_\xi(G):=\sqrt{m!}\,\langle G,\eta\Psi(x)\rangle\quad(G\in H_\Psi).
\]
The increment comparison is exactly the comparison of the corresponding
Gaussian $L^2$ metrics:
\begin{align*}
 \int_{H_c}|X_\xi-X_\zeta|^2\,d\gamma_{H_c}
 &=\|\eta c(x)-\theta c(y)\|_2^2\\
 &\le m!\,\|\eta\Psi(x)-\theta\Psi(y)\|_2^2\\
 &=\int_{H_\Psi}|Y_\xi-Y_\zeta|^2\,d\gamma_{H_\Psi},
\end{align*}
where the inequality is Lemma~\ref{lem:increment-comparison}.  Thus
Lemma~\ref{lem:comparison-principles}(i), applied on the signed index set
$\mathcal I$, gives
\begin{align}
 \int_{H_c} \sup_{L(x)\le L_0}|\langle g,c(x)\rangle|\,d\gamma_{H_c}(g)
 &\le \sqrt{m!}\,
 \int_{H_\Psi} \sup_{L(x)\le L_0}|\langle G,\Psi(x)\rangle|\,d\gamma_{H_\Psi}(G)\notag\\
 &\le C\sqrt{m!}\,n^{(m+1)/2},
 \label{eq:gaussian-orbit-low}
\end{align}
where the second line is \eqref{eq:Psi-low-width}.  Lemma~\ref{lem:comparison-principles}(ii) now gives, with its explicit comparison
factor $\sqrt{\pi/2}$,
\begin{align}
 \int \sup_{L(x)\le L_0}|A_{m,n}^{\varepsilon}(x)|\,d\mu(\varepsilon)
 &\le \sqrt{\frac{\pi}{2}}\,
 \int_{H_c} \sup_{L(x)\le L_0}|\langle g,c(x)\rangle|\,d\gamma_{H_c}(g)\notag\\
 &\le C\sqrt{m!}\,n^{(m+1)/2}.
 \label{eq:rademacher-low}
\end{align}
\begingroup
\color{black}
The bounds \eqref{eq:all-low-strata-open-width}, \eqref{eq:Psi-low-width}, and \eqref{eq:gaussian-orbit-low} combine at the scale
\[
 \sqrt n\, n^{m/2}\,\sqrt{m!}
 =\sqrt{m!}\,n^{(m+1)/2}.
\]
\endgroup
Equation~\eqref{eq:rademacher-low} is exactly \eqref{eq:low-loss-final}.
\end{proof}

\begin{corollary}
\label{cor:improved-global-upper}
There is an absolute constant $C_0>0$ such that, for every $m\ge3$,
\begin{equation}\label{eq:improved-global-upper}
 C_m^{\sym}\le C_0\sqrt{m!}.
\end{equation}
\end{corollary}

\begin{proof}
Assume first that $m\ge m_0$ and $n\ge m/e$.
\begingroup
\color{black}
Since $\{L(x)<L_0\}\subset\{L(x)\le L_0\}$,
Proposition~\ref{prop:regular-exceptional-low-loss} controls the first region.
On the complementary region $L(x)\ge L_0$,
Lemma~\ref{lem:parameter-compatibility} gives
$1\le L_0\le(m-3)/2$, and Lemma~\ref{lem:permanent-tail} gives
\endgroup
\begin{align}
 &\int \sup_{L(x)\ge L_0}|A_{m,n}^{\varepsilon}(x)|\,d\mu(\varepsilon)\notag\\
 &\qquad\le
 C e^{-L_0/2}
 \sqrt{
 1+(L_0+1)\log\frac{em}{L_0+1}
 }
 \sqrt{m!}\,n^{(m+1)/2}.
 \label{eq:tail-at-L0}
\end{align}
The coefficient is uniformly bounded.  Since
\[
 e^{-L_0/2}=e^{-4}e^{-T_m},
 \qquad
 e^{T_m}=\log(m+e^e),
\]
and $T_m\ge1$, we have
\[
 L_0+1=2T_m+9\le11T_m.
\]
Moreover,
\[
 \log\frac{em}{L_0+1}
 \le 1+\log(m+e^e)
 =1+e^{T_m}
 \le2e^{T_m}.
\]
Therefore
\begin{align}
 1+(L_0+1)\log\frac{em}{L_0+1}
 &\le 1+22T_m e^{T_m}
 \le C T_m e^{T_m},
 \label{eq:tail-coefficient-inside}
\end{align}
and hence
\begin{align}
 e^{-L_0/2}
 \sqrt{1+(L_0+1)\log\frac{em}{L_0+1}}
 &\le
 C e^{-T_m}\sqrt{T_m e^{T_m}}\notag\\
 &=C\sqrt{T_m}\,e^{-T_m/2}
 \le C.
 \label{eq:tail-coefficient-bound}
\end{align}
\begingroup
\color{black}
Since $T_m\to\infty$ as $m\to\infty$, the scalar coefficient in
\eqref{eq:tail-at-L0} is in fact
\[
 O\!\left(\sqrt{T_m}e^{-T_m/2}\right)=o(1).
\]
\endgroup

\begingroup
\color{black}
Write the exact partition
\[
 \Omega_0:=\{x:L(x)<L_0\},
 \qquad
 \Omega_1:=\{x:L(x)\ge L_0\}.
\]
\endgroup
For every choice of the orbit signs,
\[
 \sup_x|A_{m,n}^{\varepsilon}(x)|
 \le
 \sup_{x\in\Omega_0}|A_{m,n}^{\varepsilon}(x)|
 +\sup_{x\in\Omega_1}|A_{m,n}^{\varepsilon}(x)|.
\]
Integrating with respect to the Rademacher signs, applying
Proposition~\ref{prop:regular-exceptional-low-loss} to the first term and
\eqref{eq:tail-at-L0} to the second, and using \eqref{eq:tail-coefficient-bound} for its scalar coefficient, gives
\begin{align*}
 \int \sup_x|A_{m,n}^{\varepsilon}(x)|\,d\mu(\varepsilon)
 &\le
 C_1\sqrt{m!}\,n^{(m+1)/2}
 +C_2\sqrt{m!}\,n^{(m+1)/2}\\
 &=(C_1+C_2)\sqrt{m!}\,n^{(m+1)/2}.
\end{align*}
Since the normalized integral of a nonnegative function is at least its minimum, there is a choice of the orbit signs for which
\[
 \sup_x|A_{m,n}^{\varepsilon}(x)|
 \le (C_1+C_2)\sqrt{m!}\,n^{(m+1)/2}.
\]

If $n<m/e$, take the symmetric form with every coefficient equal to $1$.
Its norm is $n^m$, and therefore
\[
 \frac{n^m}{n^{(m+1)/2}}
 =
 n^{(m-1)/2}
 \le
 \left(\frac me\right)^{(m-1)/2}.
\]
Also,
\[
 \log(m!)
 =\sum_{j=1}^m\log j
 \ge\int_1^m\log x\,dx
 =m\log m-m+1,
\]
so $m!\ge e(m/e)^m$.  Since $m/e>1$ for $m\ge3$,
\[
 \left(\frac me\right)^{(m-1)/2}
 \le \left(\frac me\right)^{m/2}
 \le \sqrt{m!}.
\]
For the finitely many degrees $3\le m<m_0$, Boas's endpoint estimate
\cite[Theorem~4]{Boas}, recorded in \eqref{eq:boas-intro}, applies uniformly
in the dimension.  Enlarging the absolute constant $C_0$ completes the proof.
\end{proof}

\begin{proof}[Proof of Theorem~\ref{thm:main}]
Corollary~\ref{cor:improved-global-upper} and Proposition~\ref{prop:lower}
give
\[
 \left(\sqrt{\frac2e}+o(1)\right)\frac{\sqrt{m!}}m
 \le C_m^{\sym}\le C_0\sqrt{m!}.
\]
Taking logarithms,
\[
 \frac12\log(m!)-\log m+O(1)
 \le \log C_m^{\sym}
 \le \frac12\log(m!)+O(1),
\]
and therefore
\[
 \log C_m^{\sym}=\frac12\log(m!)+O(\log m).
\]
Stirling's formula, in the form
\[
 \frac1m\log(m!)=\log m-1+O\!\left(\frac{\log m}{m}\right),
\]
then yields
\[
 \frac1m\log C_m^{\sym}
 =\frac12(\log m-1)+o(1),
\]
which is equivalent to
\[
 (C_m^{\sym})^{1/m}\sim\sqrt{\frac me}.
\]
This proves \eqref{eq:main-log-numbered}.
\end{proof}

\begingroup
\color{black}
\section{A test of block restrictions}\label{sec:block-restrictions}

One may also identify variables only within prescribed blocks and apply
Parseval on the resulting product cube.
\begingroup
\color{black}
Retaining only the square-free Walsh level in each block does not improve the asymptotic factor $m^{-1}$ furnished by the full diagonal.
\endgroup

Let
\[
 m=b_1+\cdots+b_k,
 \qquad b_j\ge1.
\]
Partition the $m$ slots into blocks of sizes $b_1,\ldots,b_k$ and define
\[
 Q:(\R^n)^k\longrightarrow\R,
 \qquad
 Q(y^{(1)},\ldots,y^{(k)})
 :=
 A_{m,n}\bigl(
 \underbrace{y^{(1)},\ldots,y^{(1)}}_{b_1},
 \ldots,
 \underbrace{y^{(k)},\ldots,y^{(k)}}_{b_k}
 \bigr)
\]
on $(\{-1,1\}^n)^k$.

\begin{lemma}\label{lem:block-squarefree}
Let $E_j\subset\{1,\ldots,n\}$ with $|E_j|=b_j$ for
$1\le j\le k$.  The Walsh coefficient of $Q$ associated with
\[
 \chi_{E_1}(y^{(1)})\cdots\chi_{E_k}(y^{(k)})
\]
has modulus
\begin{equation}\label{eq:block-coefficient}
 \prod_{j=1}^k b_j!.
\end{equation}
\end{lemma}

\begin{proof}
In the $j$th block, a square-free character of degree $b_j$ can only arise
when each coordinate of $E_j$ occurs once.  There are $b_j!$ orderings in
that block.  The orderings in different blocks may be chosen independently,
so the prescribed product character receives $\prod_jb_j!$ contributions.
All corresponding ordered tuples have the same combined multiset of indices.
Symmetry therefore gives them the same unimodular coefficient, which proves
\eqref{eq:block-coefficient}.
\end{proof}

\begin{proposition}\label{prop:block-walsh}
For every partition $m=b_1+\cdots+b_k$ and every
$n\ge\max_jb_j$,
\begin{equation}\label{eq:block-walsh}
 C_m^{\sym}
 \ge
 \frac{
 \displaystyle
 \prod_{j=1}^k b_j!\,
 \sqrt{\prod_{j=1}^k\binom n{b_j}}
 }{
 n^{(m+1)/2}
 }.
\end{equation}
\end{proposition}

\begin{proof}
Parseval on the product cube, with normalized counting measure and
restricted to the characters in Lemma~\ref{lem:block-squarefree}, gives
\[
 \int |Q|^2\,d\mu
 \ge
 \left(\prod_{j=1}^k b_j!\right)^2
 \prod_{j=1}^k\binom n{b_j}.
\]
Since the restriction is taken on cube vertices,
\[
 \|A_{m,n}\|\ge\|Q\|_\infty\ge\|Q\|_2.
\]
Divide by $n^{(m+1)/2}$ and minimize over symmetric forms.
\end{proof}

Could a judicious choice of blocks improve the diagonal estimate?  At the
level of the square-free Walsh spectrum, the answer is no.

\begin{proposition}\label{prop:block-optimal}
For a partition $b=(b_1,\ldots,b_k)$, set
\[
 S(b):=\sum_{j=1}^k b_j(b_j-1).
\]
If $S(b)>0$, then
\begin{equation}\label{eq:block-certificate-upper}
 \sup_{n\ge\max_j b_j}
 \frac{
 \displaystyle
 \prod_{j=1}^k b_j!\,
 \sqrt{\prod_{j=1}^k\binom n{b_j}}
 }{
 n^{(m+1)/2}
 }
 \le
 \sqrt{\frac2e}\,
 \sqrt{\frac{\prod_{j=1}^k b_j!}{S(b)}}
 \le
 \sqrt{\frac2e}\,
 \sqrt{\frac{m!}{m(m-1)}}.
\end{equation}
If $S(b)=0$, then every $b_j=1$ and the corresponding certificate is at
most $1$.  Consequently, the one-block partition $b_1=m$ is
asymptotically optimal among all block Walsh restrictions.
\end{proposition}

\begin{proof}
For $b\ge1$,
\[
 b!^2\binom nb
 =
 b!\,(n)_b
 =
 b!\,n^b
 \prod_{r=0}^{b-1}\left(1-\frac rn\right).
\]
The square of the expression in \eqref{eq:block-certificate-upper} is
therefore
\[
 \frac{\prod_{j=1}^k b_j!}{n}
 \prod_{j=1}^k\prod_{r=0}^{b_j-1}
 \left(1-\frac rn\right).
\]
Since $\log(1-t)\le-t$,
\[
 \prod_{j=1}^k\prod_{r=0}^{b_j-1}
 \left(1-\frac rn\right)
 \le
 \exp\left\{-\frac{S(b)}{2n}\right\}.
\]
Thus
\[
 \left[
 \frac{
 \displaystyle
 \prod_{j=1}^k b_j!\,
 \sqrt{\prod_{j=1}^k\binom n{b_j}}
 }{
 n^{(m+1)/2}
 }
 \right]^2
 \le
 \frac{\prod_jb_j!}{n}
 e^{-S(b)/(2n)}.
\]
Since $S=S(b)>0$, we may allow $n$ to range over all positive real numbers.
The function $t^{-1}e^{-S/(2t)}$ is maximized at $t=S/2$, with maximum
$2/(eS)$.
This proves the first inequality in \eqref{eq:block-certificate-upper}.

For the second, put
\[
 M:=\frac{m!}{\prod_{j=1}^k b_j!}.
\]
Then
\[
 \frac{M S(b)}{m(m-1)}
 =
 \sum_{\substack{1\le j\le k\\ b_j\ge2}}
 \frac{(m-2)!}{(b_j-2)!\prod_{i\ne j}b_i!}.
\]
Every nonzero summand is a positive integer multinomial coefficient.  If
$S(b)>0$, at least one occurs, and hence
\[
 \frac{M S(b)}{m(m-1)}\ge1.
\]
Equivalently,
\[
 \frac{\prod_jb_j!}{S(b)}
 \le
 \frac{m!}{m(m-1)}.
\]
If $S(b)=0$, all blocks have size one and \eqref{eq:block-walsh} reduces to
$n^{-1/2}\le1$.

Finally, for the one-block partition $b_1=m$, the calculation in the proof
of Proposition~\ref{prop:lower} gives
\[
 \sup_{n\ge m}
 \frac{m!\sqrt{\binom nm}}{n^{(m+1)/2}}
 \ge
 \left(\sqrt{\frac2e}+o(1)\right)\frac{\sqrt{m!}}m
 =
 \left(1+o(1)\right)
 \sqrt{\frac2e}\sqrt{\frac{m!}{m(m-1)}}.
\]
Together with \eqref{eq:block-certificate-upper}, this shows that the
one-block partition asymptotically attains the universal upper envelope.
\end{proof}

\begin{remark}[The block-Walsh barrier]\label{rem:block-barrier}
Proposition~\ref{prop:block-optimal} does not determine the true order of
$C_m^{\sym}/\sqrt{m!}$.  It determines the reach of a method: no
identification of the variables into blocks, followed by retention of the
square-free Walsh level in each block, can improve the factor $m^{-1}$.
A sharper lower bound must use a different portion of the Walsh spectrum or
a genuinely different obstruction.
\end{remark}

\begin{remark}[Relation with the complex diagonal argument]\label{rem:BK}
Boas and Khavinson
\cite[proof of Theorem~2 and Remark~2]{BoasKhavinson}
used the complete diagonal
coefficient family $|c_\alpha|=m!/\alpha!$ together with the torus
$L_2$ norm in the complex polynomial setting.  That argument does not
directly control the real cube norm considered here.  Proposition
\ref{prop:parseval} is its discrete real counterpart: it retains only the
square-free Walsh level, yet already detects the factorial scale after the
dimension is optimized.
\end{remark}

\endgroup

\section{Proof of Theorem~B}\label{sec:unrestricted-rate}

Recall that $h_+(n)$ is the least Hadamard order not smaller than $n$, and
write $d_+(n):=h_+(n)-n$.  The construction uses the following quantitative gap estimate.

\begin{proposition}\label{prop:hadamard-gap}
There is an absolute constant $K_{\mathrm H}>0$ such that
\begin{equation}\label{eq:hadamard-gap}
 d_+(n)\le K_{\mathrm H}n^{1/6}
 \qquad(n\ge1).
\end{equation}
\end{proposition}

\begin{proof}
Brent, Osborn and Smith record
$\gamma(H)=O(H^{1/6})$ for the Hadamard gap function, as the consequence of
their Lemma~13 and Livinskyi's construction; see
\cite[\S~6, Lemma~13 and Eq.~(17)]{BrentOsbornSmithLong}.
If $H<H'$ are consecutive Hadamard orders and $H\le n<H'$, then
\[
 h_+(n)-n\le H'-H=O(H^{1/6})=O(n^{1/6}).
\]
This proves the estimate for all sufficiently large $n$; increasing the
constant accounts for the remaining values.
\end{proof}

The Hadamard chain of \cite[Lemma~2.1]{PellegrinoRaposo} can be
truncated directly in the original rectangular dimensions.  This avoids
enlarging the first variable and yields the product in Theorem~B.  Combining
the construction with the quantitative Hadamard gap in Proposition~\ref{prop:hadamard-gap} makes the
dependence on both the degree and the dimensions explicit.

\begin{lemma}\label{lem:block-ksz}
Let $m\ge2$ and $1\le n_1\le\cdots\le n_m$.  For each
$k=2,\ldots,m$, let $t_k\ge n_k$ be a Hadamard order.  There is a real sign
$m$-linear form
\[
 A:\ell_\infty^{n_1}\times\cdots\times\ell_\infty^{n_m}
 \longrightarrow\R
\]
such that
\begin{equation}\label{eq:block-est}
 \|A\|\le\sqrt{n_1n_m}\prod_{k=2}^m\sqrt{t_k}.
\end{equation}
\end{lemma}

\begin{proof}
For $k=2,\ldots,m$, choose a Hadamard matrix
$H^{(k)}=(h^{(k)}_{ij})_{1\le i,j\le t_k}$, and let $B^{(k)}$ be its
submatrix consisting of the first $n_{k-1}$ rows and first $n_k$ columns.
Zero extension followed by coordinate restriction gives
\begin{equation}\label{eq:rectangular-hadamard-operator}
 \|B^{(k)}\|_{2\to2}\le\|H^{(k)}\|_{2\to2}=\sqrt{t_k}.
\end{equation}
Define
\[
 A(x^{(1)},\ldots,x^{(m)})
 :=
 \sum_{i_1=1}^{n_1}\cdots\sum_{i_m=1}^{n_m}
 \left(\prod_{k=2}^m h^{(k)}_{i_{k-1},i_k}\right)
 x^{(1)}_{i_1}\cdots x^{(m)}_{i_m}.
\]
Every coefficient is a sign.

Fix $\|x^{(k)}\|_\infty\le1$.  Put $v^{(m)}=x^{(m)}$ and, for
$k=m-1,\ldots,2$, define
\[
 v^{(k)}_{i_k}
 :=x^{(k)}_{i_k}
 \sum_{i_{k+1}=1}^{n_{k+1}}
 h^{(k+1)}_{i_k,i_{k+1}}v^{(k+1)}_{i_{k+1}}.
\]
Coordinatewise multiplication by $x^{(k)}$ is contractive on $\ell_2$, so
\[
 \|v^{(k)}\|_2
 \le\sqrt{t_{k+1}}\,\|v^{(k+1)}\|_2
\]
by \eqref{eq:rectangular-hadamard-operator}.  At the first variable,
\begin{align*}
 |A(x^{(1)},\ldots,x^{(m)})|
 &=|\langle x^{(1)},B^{(2)}v^{(2)}\rangle|\\
 &\le\sqrt{n_1}\sqrt{t_2}\,\|v^{(2)}\|_2\\
 &\le\sqrt{n_1}
       \left(\prod_{k=2}^m\sqrt{t_k}\right)\|x^{(m)}\|_2\\
 &\le\sqrt{n_1n_m}\prod_{k=2}^m\sqrt{t_k}.
\end{align*}
The same computation includes $m=2$, when the intermediate recursion is
empty.
\end{proof}

\begin{proof}[Proof of Theorem~\ref{thm:unrestricted-summary}]
For $k=2,\ldots,m$, put $t_k:=h_+(n_k)$.  Lemma~\ref{lem:block-ksz} gives
\[
 S_{m,\mathbf n}
 \le\sqrt{n_1n_m}\prod_{k=2}^m\sqrt{h_+(n_k)}.
\]
Since
\[
 n_m^{1/2}\prod_{j=1}^m n_j^{1/2}
 =\sqrt{n_1n_m}\prod_{k=2}^m\sqrt{n_k},
\]
division yields the first inequality in
\eqref{eq:unrestricted-rectangular}.  Proposition~\ref{prop:hadamard-gap}
and $\log(1+u)\le u$ give
\begin{align*}
 \log\prod_{k=2}^m
 \left(\frac{h_+(n_k)}{n_k}\right)^{1/2}
 &\le\frac12\sum_{k=2}^m
 \log\bigl(1+K_{\mathrm H}n_k^{-5/6}\bigr)\\
 &\le\frac{K_{\mathrm H}}2\sum_{k=2}^m n_k^{-5/6}.
\end{align*}
This proves the second inequality and, in equal dimensions,
\eqref{eq:two-parameter-main}.
\end{proof}

\begin{proof}[Proof of Corollary~\ref{cor:price-symmetry}]
Put $q_m:=\lfloor m^2/2\rfloor$ and $N_m:=h_+(q_m)$.  By
Proposition~\ref{prop:hadamard-gap},
\[
 N_m-q_m=O(q_m^{1/6})=O(m^{1/3}),
\]
so $N_m=m^2/2+O(m^{1/3})$.  Since $N_m$ is a Hadamard order,
Lemma~\ref{lem:block-ksz} in equal dimensions gives
\[
 S_{m,N_m}\le N_m^{(m+1)/2}.
\]

For the symmetric problem, Proposition~\ref{prop:parseval} gives
\[
 \frac{S_{m,N_m}^{\sym}}{N_m^{(m+1)/2}}
 \ge
 \frac{\sqrt{m!}}{\sqrt{N_m}}
 \left[\prod_{k=0}^{m-1}
 \left(1-\frac{k}{N_m}\right)\right]^{1/2}.
\]
Here $N_m\sim m^2/2$ and $\max_{k<m}k/N_m=O(m^{-1})$, whence
\begin{align*}
 \sum_{k=0}^{m-1}\log\left(1-\frac{k}{N_m}\right)
 &=-\frac1{N_m}\sum_{k=0}^{m-1}k
   +O\left(\frac1{N_m^2}\sum_{k=0}^{m-1}k^2\right)\\
 &=-\frac{m(m-1)}{2N_m}+O\left(\frac{m^3}{N_m^2}\right)\\
 &=-1+O(m^{-1}).
\end{align*}
Since $N_m^{-1/2}=(\sqrt2+o(1))/m$, this proves
\eqref{eq:quadratic-symmetric-lower}.  Moreover,
$S_{m,N_m}\le N_m^{(m+1)/2}$, so the same lower bound holds after dividing
the symmetric minimum by the unrestricted one.  This is the final assertion
of the corollary.
\end{proof}

\section*{Acknowledgments and funding}

D. Pellegrino was supported by CNPq Grants No.~406457/2023-9
(CNPq/MCTI N\textsuperscript{o}~10/2023), No.~403964/2024-5
(MCTI/CNPq N\textsuperscript{o}~16/2024), and No.~305807/2025-0.
E. Teixeira was supported by the Grayce B. Kerr Chair funds at Oklahoma
State University.

This research was conducted in part under the DARPA ExpMath project
\emph{``A Human-Centered Framework for AI-Mathematician Collaboration in
Research-Level Mathematics''} (Agreement No.~HR0011262E029), in which
E.~Teixeira serves as a co-principal investigator.  He gratefully acknowledges
DARPA's support.
The views, opinions, and findings expressed here are those of the authors and
should not be interpreted as representing the official views or policies of
the Department of Defense or the U.S.\ Government.

The \textsc{Lea} Prover was used to check several proofs and their
dependencies, to compare alternative formulations, and to assist with
routine \LaTeX.  The authors independently verified its output and take
responsibility for the mathematical and expository content.

\section*{Data availability}

No datasets were generated or analyzed in this work.

\section*{Competing interests}

The authors declare that they have no competing interests.

\bigskip
\noindent\textsc{Daniel M. Pellegrino}\\
Departamento de Matem\'atica, Universidade Federal da Para\'iba,
Jo\~ao Pessoa, PB, Brazil\\
\texttt{dmpellegrino@gmail.com}

\medskip
\noindent\textsc{Anselmo B. Raposo Jr.}\\
Coordena\c{c}\~{a}o do Curso de Matem\'{a}tica -- Bacharelado, Universidade Federal do Maranh\~ao,
S\~ao Lu\'is, MA, Brazil\\
\texttt{anselmo.junior@ufma.br}

\medskip
\noindent\textsc{Eduardo V. Teixeira}\\
Department of Mathematics, Oklahoma State University,
Stillwater, OK 74078, USA\\
\texttt{eduardo.teixeira@okstate.edu}


\bigskip
\begin{thebibliography}{99}

\bibitem{AlbuquerqueRezende}
N. G. Albuquerque and L. Rezende,
\emph{Asymptotic estimates for unimodular multilinear forms with small norms on sequence spaces},
Bull. Braz. Math. Soc. (N.S.) \textbf{52} (2021), no.~1, 23--39.

\bibitem{Banach}
S. Banach,
\emph{\"Uber homogene Polynome in $(L^2)$},
Studia Math. \textbf{7} (1938), 36--44.

\bibitem{BennettGoodmanNewman}
G. Bennett, V. Goodman and C. M. Newman,
\emph{Norms of random matrices},
Pacific J. Math. \textbf{59} (1975), no.~2, 359--365.


\bibitem{Boas}
H. P. Boas,
\emph{Majorant series},
in \emph{Several complex variables (Seoul, 1998)},
J. Korean Math. Soc. \textbf{37} (2000), no.~2, 321--337.

\bibitem{BoasKhavinson}
H. P. Boas and D. Khavinson,
\emph{Bohr's power series theorem in several variables},
Proc. Amer. Math. Soc. \textbf{125} (1997), no.~10, 2975--2979.

\bibitem{BrentOsbornSmithLong}
R. P. Brent, J.-A. H. Osborn and W. D. Smith,
\emph{Lower bounds on maximal determinants of binary matrices via the
probabilistic method}, long version, arXiv:1402.6817v6 [math.CO] (2016),
37 pp.; see also \emph{Probabilistic lower bounds on maximal determinants of
binary matrices}, Australas. J. Combin. \textbf{66} (2016), no.~3, 350--364.

\bibitem{CarandoRodriguez}
D. Carando and J. T. Rodr\'iguez,
\emph{Symmetric multilinear forms on Hilbert spaces: where do they attain their norm?},
Linear Algebra Appl. \textbf{563} (2019), 178--192.

\bibitem{CarlenLiebLoss}
E. A. Carlen, E. H. Lieb and M. Loss,
\emph{An inequality of Hadamard type for permanents},
Methods Appl. Anal. \textbf{13} (2006), no.~1, 1--18.

\bibitem{DefantMastylo}
A. Defant and M. Masty\l o,
\emph{Aspects of the Kahane--Salem--Zygmund inequalities in Banach spaces},
Rev. R. Acad. Cienc. Exactas F\'is. Nat. Ser. A Mat. RACSAM \textbf{117} (2023), no.~1, Paper No.~44, 40 pp.



\bibitem{FriedlandWang}
S. Friedland and L. Wang,
\emph{Spectral norm of a symmetric tensor and its computation},
Math. Comp. \textbf{89} (2020), no.~325, 2175--2215.

\bibitem{Kahane}
J.-P. Kahane,
\emph{Some Random Series of Functions}, 2nd ed.,
Cambridge Studies in Advanced Mathematics 5,
Cambridge University Press, Cambridge, 1985.

\bibitem{LedouxGaussian}
M. Ledoux,
\emph{Isoperimetry and Gaussian analysis},
in P. Bernard (ed.), \emph{Lectures on Probability Theory and Statistics
(Saint-Flour, 1994)}, Lecture Notes in Math. 1648,
Springer-Verlag, Berlin, 1996, 165--294.

\bibitem{ManteroTonge}
A. M. Mantero and A. M. Tonge,
\emph{The Schur multiplication in tensor algebras},
Studia Math. \textbf{68} (1980), no.~1, 1--24.

\bibitem{MastyloSzwedek}
M. Masty\l o and R. Szwedek,
\emph{Kahane--Salem--Zygmund polynomial inequalities via Rademacher processes},
J. Funct. Anal. \textbf{272} (2017), no.~11, 4483--4512.

\bibitem{PellegrinoRaposo}
D. Pellegrino and A. Raposo Jr.,
\emph{Constants of the Kahane--Salem--Zygmund inequality asymptotically bounded by $1$},
J. Funct. Anal. \textbf{282} (2022), no.~2, Paper No.~109293, 21 pp.

\bibitem{SalemZygmund}
R. Salem and A. Zygmund,
\emph{Some properties of trigonometric series whose terms have random signs},
Acta Math. \textbf{91} (1954), 245--301.

\bibitem{Varopoulos}
N. Th. Varopoulos,
\emph{On an inequality of von Neumann and an application of the metric theory
of tensor products to operators theory},
J. Funct. Anal. \textbf{16} (1974), no.~1, 83--100.

\bibitem{Vitale}
R. A. Vitale,
\emph{Some comparisons for Gaussian processes},
Proc. Amer. Math. Soc. \textbf{128} (2000), no.~10, 3043--3046.

\end{thebibliography}
\end{document}